\documentclass[twoside,11pt,reqno]{amsart}
\usepackage{amsmath,amssymb,enumerate}
\usepackage{verbatim}

\newcommand{\CC}{{\mathbb{C}}}
\newcommand{\FF}{{\mathbb{F}}}

\newcommand{\cE}{{\mathcal{E}}}

\newcommand{\Id}{{\operatorname{Id}}}
\newcommand{\Irr}{{\operatorname{Irr}}}

\newcommand{\reg}{{\operatorname{reg}}}
\newcommand{\Syl}{{\operatorname{Syl}}}

\newcommand\GL{{\operatorname{GL}}}
\newcommand\PGL{{\operatorname{PGL}}}
\newcommand\SL{{\operatorname{SL}}}
\newcommand\PSL{{\operatorname{PSL}}}
\newcommand\GO{{\operatorname{GO}}}
\newcommand\PGO{{\operatorname{PGO}}}
\newcommand\SO{{\operatorname{SO}}}
\newcommand\POm{{\operatorname{P\Omega}}}
\newcommand\PGU{{\operatorname{PGU}}}
\newcommand\PSU{{\operatorname{PSU}}}
\newcommand\SU{{\operatorname{SU}}}
\newcommand\PSp{{\operatorname{PSp}}}
\newcommand\Sp{{\operatorname{Sp}}}

\newcommand\Spin{{\operatorname{Spin}}}
\newcommand\GU{{\operatorname{GU}}}

\newcommand{\up}{^{-1}}
\newcommand{\tw}[1]{{}^#1\!}
\newcommand{\diag}{{\rm diag}}
\newcommand{\flhalfn}{{\lfloor n/2\rfloor}}

\let\lan=\langle
\let\ran=\rangle

\let\al=\alpha
\let\eps=\varepsilon
\let\lam=\lambda
\let\om=\omega

\let\si=\sigma
\let\ep=\varepsilon

\def\ir{irreducible }

\newtheorem{thm}{Theorem}[section]
\newtheorem{lem}[thm]{Lemma}
\newtheorem{cor}[thm]{Corollary}
\newtheorem{prop}[thm]{Proposition}

\theoremstyle{definition}
\newtheorem{rem}[thm]{Remark}

\begin{document}

\title{Steinberg-like characters for finite simple groups}

\date{\today}

\author{Gunter Malle}
\address{FB Mathematik, TU Kaiserslautern, Postfach 3049,
  67653 Kaisers\-lautern, Germany}
  \makeatletter\email{malle@mathematik.uni-kl.de}\makeatother
\author{Alexandre Zalesski}
\address{Department of Physics, Informatics and Mathematics, National Academy
  of Sciences of Belarus, Minsk, Belarus}
  \makeatletter\email{alexandre.zalesski@gmail.com}\makeatother

\thanks{The first author gratefully acknowledges financial support by SFB TRR 195.}

\keywords{characters of projective indecomposable modules, Steinberg-like characters}

\subjclass[2010]{20C15,20C30,20C33}

\begin{abstract}
Let $G$ be a finite group and, for a prime $p$, let $S$ be a Sylow
$p$-subgroup of $G$. A character $\chi$ of $G$ is called $\Syl_p$-regular if
the restriction of $\chi$ to $S$ is the character of the regular
representation of $S$. If, in addition, $\chi$ vanishes at all elements of
order divisible by $p$, $\chi$ is said to be Steinberg-like. For every finite
simple group $G$ we determine all primes $p$ for which $G$ admits a
Steinberg-like character, except for alternating groups in characteristic~2.
Moreover, we determine all primes for which $G$ has a projective $FG$-module
of dimension $|S|$, where $F$ is an algebraically closed field of
characteristic~$p$.
\end{abstract}

\maketitle



\section{Introduction}
Let $G$ be a finite group and, for a prime $p$, let $S$ be a Sylow $p$-subgroup
of $G$. A character $\chi$ of $G$ is called $\Syl_p$-\emph{vanishing} if
$\chi(u)=0$ for every $1\neq u\in S$; and if, additionally, $\chi(1)=|S|$
then we say that $\chi$ is $\Syl_p$-\emph{regular}.
If $\chi(g)=0$ whenever $|g|$ is divisible by $p$ then $\chi$ is called
$p$-\emph{vanishing}; and if, additionally, $\chi(1)=|S|$ then we say that
$\chi$ is \emph{Steinberg-like}. Steinberg-like and Syl$_p$-regular characters
for Chevalley groups in defining characteristic $p$ are studied in \cite{PZ}.
Specifically, for all simple groups of Lie type in characteristic $p$ except
$B_n(q)$, $n=3,4,5$, and $D_n(q)$, $n=4,5$, the Steinberg-like characters
for the prime $p$ have been determined in \cite{PZ}.

Our main motivation to study this kind of characters is their connection
with characters of projective indecomposable modules. The study of projective
indecomposable modules of dimension $|S|$ was initiated by Malle and Weigel
\cite{MW}; they obtained a full classification of such modules for arbitrary
finite simple groups $G$ assuming that the character of the module has the
trivial character $1_G$ as a constituent. In \cite{Z11}, this restriction was
removed for simple groups of Lie type with defining characteristic $p$. Some
parts of the proofs there were valid not only for characters of projective
modules, but also for Steinberg-like or even $\Syl_p$-regular characters.

In this paper we complete the classification of projective indecomposable
modules of dimension $|S|$ for simple groups $G$.
The first main result is a classification of Steinberg-like characters for
simple groups, with the sole exception of alternating groups for the prime
$p=2$:

\begin{thm}   \label{th1}
 Let $G$ be a finite non-abelian simple group, $p$ a prime dividing $|G|$ and
 let $\chi$ be a Steinberg-like character of $G$ with respect to~$p$. Then one
 of the following holds:
 \begin{itemize}
  \item[\rm(1)] $\chi$ is irreducible, and the triple $(G,p,\chi(1))$ is as in
   Proposition~$\ref{prop:ir2}$;
  \item[\rm(2)] Sylow $p$-subgroups of $G$ are cyclic and $(G,p,\chi(1))$ is as
   in Proposition~$\ref{ms4}$;
  \item[\rm(3)] $G$ is of Lie type in characteristic~$p$ (see {\rm\cite{PZ}});
  \item[\rm(4)] $p=2$ and $G=\PSL_2(q)$ with $q+1=2^k$; or
  \item[\rm(5)] $p=2$ and $G=A_n$, $n\ge8$.
 \end{itemize}
\end{thm}

In fact, in many instances we even classify all $\Syl_p$-regular characters.
Examples for case~(5) when $n=2^k$ or $2^{k}+1$ are presented in
Corollaries~\ref{so2} and~\ref{so39}. We are not aware of any further examples.

Our second main result determines reducible projective modules of simple
groups of minimal possible dimension~$|G|_p$.

\begin{thm}   \label{th2}
 Let $G$ be a finite non-abelian simple group, $p$ a prime dividing $|G|$ and
 $S$ a Sylow $p$-subgroup of $G$. Then $G$ has a reducible projective
 $\overline{\FF}_pG$-module of dimension $|S|$ if and only if one of the
 following holds:
\begin{enumerate}
  \item[$(1)$] $G=\PSL_2(q)$, $q>4$, $|S|=q+1$;
  \item[$(2)$] $G=\PSL_n(q)$, $n$ is an odd prime, $n{\not|}(q-1)$,
   $|S|=(q^n-1)/(q-1)$;
  \item[$(3)$] $G=A_p$, $|S|=p\ge5$;
  \item[$(4)$] $G=M_{11}$, $|S|=11$; or
  \item[$(5)$] $G=M_{23}$, $|S|=23$.
 \end{enumerate}
\end{thm}

Note that irreducible projective $\overline{\FF}_pG$-modules of dimension
$|G|_p$ are in bijection with \ir characters of defect 0 of that degree,
listed in Proposition~\ref{prop:ir2} for simple groups.
\smallskip

The paper is built up as follows. After some preliminaries we recall the
classification of irreducible Steinberg-like characters in
Section~\ref{sec:ic1} (Proposition~\ref{prop:ir2}). In Section~\ref{sec:cycSyl}
we classify $\Syl_p$-regular characters in the case of cyclic Sylow
$p$-subgroups (Proposition~\ref{ms4}), in Section~\ref{sec:spor} we treat the
sporadic groups (Theorem~\ref{thm:spor}). The alternating groups are handled
in Section~\ref{sec:alt} (Theorem~\ref{thm:An} for $p$ odd, and in
Section~\ref{subsec:alt p=2} some partial results for $p=2$, see
Theorems~\ref{thm:An p=2 ind} and~\ref{22a}). The exceptional groups of Lie
type are considered in Section~\ref{sec:exc} (Theorem~\ref{thm:exc}).
The rest of our paper deals with the classical groups of Lie type. We start off
in Section~\ref{sec:def char} by ruling out the remaining possibilities in
defining characteristic from \cite{PZ}. The case of large Sylow $p$-subgroups
for non-defining primes $p$ is settled in Section~\ref{sec:class large}. In
Section~\ref{sec:class p>2} we discuss the small cases when $p>2$, while the
proof of our main theorems is achieved in Section~\ref{sec:class p=2} by
treating the case when $p=2$.

\section{Preliminaries}

We start off by fixing some notation. Let $\FF_q$ be the finite field of
$q$ elements and $\overline{\FF}_q$ an algebraic closure of $\FF_q$. The
cardinality of
a set $X$ is denoted by $|X|$. The greatest common divisor of integers $m,n$
is denoted by $(m,n)$; if $p$ is a prime then $|n|_p$ is the $p$-part of $n$,
that is, $n=|n|_pm$, where $(m,p)=1$. If $(m,n)=m$, we write $m|n$.

For a finite group $G$, $\Irr(G)$ is the set of its \ir characters and
$\Irr_1(G)$ is the set of all linear characters of $G$ (that is, of degree 1).
We denote by $1_G$ the trivial character and by $\rho^\reg_G$ the regular
character of $G$. We write $S\in\Syl_p(G)$ to mean that $S$ is a Sylow
$p$-subgroup of $G$. A group of order coprime to $p$ is called a $p'$-group.
Further, $Z(G)$, $G'$ denote the center and the derived subgroup of $G$,
respectively.

If $H$ is a subgroup of $G$ then $C_G(H)$, $N_G(H)$ denote the centraliser and
normaliser of $H$ in $G$, respectively. If $\chi$ is a character of $G$ then
we write $\chi|_H$ for the restriction of $\chi$ to $H$. The
$H$-\emph{level of} $\chi$ is the maximal integer $l\geq0$ such that
$\chi|_H-l\cdot \rho^\reg_H$ is a proper character of $H$.
If a prime $p$ is fixed then the $p$-\emph{level} $l_p(\chi)$ of $\chi$ is
the $S$-level of $\chi$ for $S\in \Syl_p(G)$. (For quasi-simple groups with
cyclic Sylow $p$-subgroups irreducible characters of $p$-level $l=1,2$ are
studied in \cite{Z99,RZ}, respectively.) The inner product of characters
$\lam,\mu$ of $G$ is denoted by $(\lam,\mu)$, sometimes by $(\lam,\mu)_G$.
The character of $G$ induced from a character $\mu$ of $H$ is denoted by
$\mu^H$.

Let $P\le G$ be finite groups, $N$ a normal subgroup of $P$ and $L=P/N$.
Let $F$ be a field and  $M$ an $FG$-module. Then $M^N:=C_M(N)$ becomes an
$FL$-module, which is called \emph{generalised restriction of $M$ to} $L$ and
denoted by $r^G_{P/N} M$ in \cite[\S 70A, p.~667]{CR2}. If $\beta$ is the
Brauer (or ordinary) character of $M$ then we also write $r^G_{P/N} \beta$ for
the Brauer (or ordinary) character of $L$ afforded by $M^N$.

Let $e=e_p(q)$ ($p>2$, $(p,q)=1$) be the minimal integer $i>0$ such that
$q^i-1$ is divisible by $p$. If $p=2$ and $q$ is odd then we set $e_2(q)=1$
if $4|(q-1)$ and $e_2(q)=2$ if $4|(q+1)$.
\medskip

The next two lemmas follow from the definitions; here $G$ is a finite group and
$S\in\Syl_p(G)$.

\begin{lem}   \label{lem:a51}
 Let $\chi$ be a $\Syl_p$-regular character of $G$. Then every linear character
 occurs in $\chi|_S$ with multiplicity~$1$. In particular, $(\chi|_S,1_S)= 1$.
 If $S$ is abelian then $\chi|_S$ is multiplicity free.
\end{lem}

\begin{proof}
As $\chi|_S=\rho_S^\reg$, this follows from the corresponding properties of
$\rho_S^\reg$.
\end{proof}

\begin{lem}   \label{pr1}
 Let $G=G_1\times G_2$ be a direct product, and let $\chi_1,\chi_2$ be
 irreducible characters of $G_1,G_2$ respectively. Then the $p$-level of
 $\chi_1\otimes\chi_2$ is the product of the $p$-levels of $\chi_1$ and
 $\chi_2$.
\end{lem}

\begin{lem}   \label{lem:gt7}
 Let $N$ be a $p'$-subgroup of $G$ normalised by $S$. Let $\chi$ be a
 faithful Steinberg-like character of $G$. Then $N$ is abelian and
 $C_G(S)= Z(G) Z(S)$.
\end{lem}

\begin{proof}
Let $H=NS$. Then $\chi|_H$ is Steinberg-like. As $H$ is $p$-solvable, every
$p$-vanishing character is the character of a projective module
\cite[Lemma 10.16]{Nv}. As $\chi(1)=|S|$, the module in question is
indecomposable. Then $\chi|_H$ is induced from an irreducible character $\al$,
say, of $N$ \cite[Thm.~10.13]{Nv}.
As $\al^H(1)=\al(1)\cdot|H:N|=\al(1)\cdot|S|$ and $\chi(1)=|S|$, it follows
that $\al(1)=1$. Let $N'$ be the derived subgroup of $N$. Then $N'$ is normal
in $H$ and $\al(N')=1$. Therefore, $\al^H|_{N'}=|S|\cdot 1_{N'}$, that is,
$N'$ lies in the kernel of $\al^H$. Since $\chi$ and hence $\chi|_H=\al^H$
is faithful, we have $N'=1$. So $N$ is abelian as claimed.

Note that $C_G(S)=A\times Z(S)$, where $A$ is a $p'$-group. Take $N=A$ above,
so $H=A\times S$. So now $[N,S]=1$ and $N$ is abelian. It follows that in any
representation afforded by $\al^H$, $N$ consists of scalar matrices.
As $\chi$ is faithful, we have $[N,G]=1$, as required.
\end{proof}

Thus, if $G$ is a simple group then $C_G(S)=Z(S)$ is a necessary condition for
$G$ to have a Steinberg-like character.

\begin{rem}
A $p'$-subgroup $N$ normalised by a Sylow $p$-subgroup of $G$ is
called a \emph{$p$-signaliser} in the theory of finite groups. Thus,
Lemma~\ref{lem:gt7} tells us that if $G$ admits a faithful Steinberg-like
character then every $p$-signaliser is abelian, and $C_G(S)=Z(G)Z(S)$.
\end{rem}

\begin{lem}   \label{hc1}
 Let $G$ be a finite group, $P$ a subgroup with $(|G:P|,p)=1$, $U$ a normal
 $p$-subgroup of $P$ and let $L=P/U$. Let $T,S$ be Sylow $p$-subgroups of
 $L,G$, respectively. Let $\chi$ be a character of $G$ and
 $\lam=r_{P/U}^G(\chi)$.
 \begin{enumerate}
  \item[\rm(a)] If $\chi|_S=m\cdot\rho_S^\reg$ then $\lam|_T=m\cdot\rho_T^\reg$.
   In other words, $l_p(\chi)=l_p(\lam)$. In particular, if $\chi$ is
   $\Syl_p$-regular then so is $\lam$.
  \item[\rm(b)] If $\chi$ is a $p$-vanishing character of $G$ then $\lam$ is a
   $p$-vanishing character of $L$.
  \item[\rm(c)] Let $K:=O^{p'}(L)$. If $\chi$ is a Steinberg-like or
   $Syl_p$-regular character of $G$ then so is the character $\lam|_K$ of $K$.
 \end{enumerate}
\end{lem}

\begin{proof}
We can assume that $S\le P$ and $T=S/U$. \\
(a) As $\chi|_S=m\cdot \rho_S^\reg$, it follows that $\lam|_{T}$ coincides with
$m\cdot \rho_T^\reg$, whence the claim. \\
(b) We have to show that $\lam$ vanishes at all $p$-singular elements of $L$.
Let $M$ be a $\CC G$-module afforded by $\chi$.
Then $C_M(U)=\{\frac{1}{|U|}\sum_{u\in U} ux\mid x\in M\}$.
Observe that if $g\in P$ has projection to $L$ which is not a $p'$-element,
then $gu$ is not a $p'$-element for any $u\in U$. Thus, for any such element
$g$, it follows that $\lam(g)=\frac{1}{|U|}\sum_{u\in U}\chi(gu)=0$ by
assumption, whence the claim.\\
(c) Obvious.
\end{proof}

\begin{lem}   \label{da3}
 Let $G=G_1\times G_2$ be a direct product. Suppose that $l_p(\si)\geq k$ for
 every non-zero $\Syl_p$-vanishing (resp., $p$-vanishing) character $\si$ of
 $G_2$. Then $l_p(\chi)\geq k$ for every $\Syl_p$-vanishing
 (resp., $p$-vanishing) character $\chi$ of $G$.
\end{lem}

\begin{proof}
Let $S_1\in\Syl_p(G_1)$. Set $U=S_1$ and $P=N_G(U)$, so
$P=N_{G_1}(S_1)\times G_2$. Then $L:=P/U=L_1\times G_2$, where
$L_1=N_{G_1}(S_1)/S_1$. Let $\chi$ be a $\Syl_p$-vanishing (resp.,
$p$-vanishing) character of $G$. Let $\lam=r_{P/U}^G(\chi)$ be the generalised
restriction of $\chi$ to $L$. By Lemma~\ref{hc1}, $\lam$ is a $\Syl_p$-vanishing
(resp., $p$-vanishing) character of $L$ and $l_p(\chi)=l_p(\lam)$. Then
$l_p(\lam)=l_p(\lam|_{G_2})$, as $L_1$ is a $p'$-group.
By assumption, $l_p(\lam|_{G_2})\geq k$, whence the result.
\end{proof}

\begin{lem}   \label{nn2}
 Let $G=G_1\times G_2$, where $|G_2|_p>1$ and let $\chi$ be a $p$-vanishing
 character of $G$. Then $\chi=\sum_i \eta_i\si_i$, where $\eta_i\in\Irr(G_1)$
 are all distinct, and $\si_i$ are $p$-vanishing characters of $G_2$. In
 addition, $\chi_1:=\sum_i l_p(\si_i)\eta_i$ is a $p$-vanishing character
 of $G_1$, and $l_p(\chi_1)=l_p(\chi)$.
\end{lem}

\begin{proof}
Write $\chi=\sum_i \eta_i\si_i$, where $\eta_i\in\Irr(G_1)$ are all distinct,
and the $\si_i$'s are some characters of $G_2$ (reducible, in general).
Let $g\in G_1$, and let $x\in G_2$ be $p$-singular.
Then $0=\chi(gx)=\sum_i \eta_i(g)\si_i(x)$. As the characters $\eta_i$ are
linearly independent, it follows that $\si_i(x)=0$ for every $i$, that is, the
$\si_i$'s are $p$-vanishing.

In addition,
$|G_2|_p\sum l_p(\si_i)\eta_i=\sum_i \eta_i\si_i(1)=\chi|_{G_1}$.
So $\sum_i l_p(\si_i)\eta_i$ is $p$-vanishing. Let $l_p(\chi)=m$; then
$\chi(1)=m|G|_p=m|G_1|_p|G_2|_p=\sum\eta_i(1)\sigma_i(1)
=\sum\eta_i(1)l_p(\sigma_i)|G|_2$, whence
$m|G_1|_p=\sum\eta_i(1)l_p(\sigma_i)$, as required.
\end{proof}

\begin{cor}\label{cn2}
 Let $G=G_1\times G_2$ and $\chi$ be as in Lemma~$\ref{nn2}$, and $S_i$ a
 Sylow $p$-subgroup of $G_i$, $i=1,2$.
 Let $\eta_1,\ldots,\eta_k$ be the irreducible constituents of $\chi|_{G_1}$,
 and $\eta=\eta_1+\cdots +\eta_k$. Suppose that $l_p(\sigma)\geq m$ for every
 non-zero $p$-vanishing character $\sigma$ of $G_2$.
 Then $l_p(\chi)\ge m\cdot \eta(1)/|S_1|$.
\end{cor}

\begin{proof}
Let $\chi=\sum \eta_i\sigma _i$ be as in Lemma~\ref{nn2}. By assumption,
$\si_i|_{S_2}=m_i\cdot \rho_{S_2}^\reg$, where $m_i\geq m$.
So $m\cdot \rho_{S_2}^\reg$ is a subcharacter of $\si_i|_{S_2}$. Therefore,
$\sum_i(\eta_i|_{S_1}\cdot m\cdot \rho_{S_2}^\reg)
=(\sum_i \eta_i)|_{S_1}\cdot m\cdot \rho_{S_2}^\reg$ is a subcharacter of
$\chi|_{S_1\times S_2}$. Now
$\chi(1)\ge m\,\eta(1)\,|S_2|=m\,\eta(1)\,|G|_p/|S_1|$.
As $\chi(1)$ is a multiple of $|G|_p$, we have $\chi(1)=l_p(\chi)|G|_p$, and
the result follows.
\end{proof}

\begin{prop}   \label{ga1}
 Let $G$ be a finite group and $N\lhd G$ a normal subgroup such that $G/N$ is
 a cyclic $p$-group. Let $\chi$ be a $p$-vanishing character of $G$.
 Then:
 \begin{enumerate}
  \item[\rm(a)] $\chi=\psi^G$ for some character $\psi$ of N;
  \item[\rm(b)] if $h\in N$ is $p$-singular and the conjugacy classes of $h$
   in $G$ and in $N$ coincide then $\psi(h)=0$;
  \item[\rm(c)] if $\psi$ is G-invariant then $\psi$ is $p$-vanishing.
 \end{enumerate}
\end{prop}

\begin{proof}
(a) Let $\lam\in\Irr(G)$ be a linear character that generates $\Irr(G/N)$.
As all elements of $G\setminus N$ are $p$-singular, $\chi$ vanishes on
$G\setminus N$. It follows that $\lam\cdot\chi =\chi$. Thus, if we write
$\chi=\sum_j a_j\chi_j$ as a non-negative linear combination of irreducible
characters $\chi_j\in\Irr(G)$, then $a_j$ is constant on orbits under
multiplication with $\lam$. It clearly suffices to show the claim for a single
orbit, say $\chi=\sum_{i=1}^{p^f}\lam^i\chi'$ with $\chi'\in\Irr(G)$ and
$f$ minimal such that $\lam^{p^f}\chi'=\chi'$.   \par
Set $M:=\ker(\lam^{p^f})$. Then $\chi'|_M$ is \ir as so is $\chi'$, so
$\chi=(\chi'|_M)^G$.
Now note that $\lam^{p^f}$ generates $\Irr(G/M)$, so $\lam^{p^f}(m)\neq 1$
for $m\notin (M\setminus N)$.
Thus, as $\lam^{p^f}\chi'=\chi'$, it follows that $\chi'$ vanishes on
$M\setminus N$, and hence  $\chi'|_M=\psi^M$ is induced from some
$\psi\in\Irr(N)$. Then $\chi=(\chi'|_M)^G=(\psi^M)^G=\psi^G$ as claimed.

(b) For $g\in G$ define the character $\psi^g$ of $N$ by
$\psi^g(x)=\psi(gxg^{-1})$ ($x\in N$). It is well know that $\psi^G|_N$ is a
sum of $p^k$ characters $\psi^g$ for suitable $g\in G$. By assumption,
$\psi^g(h)=\psi(h)$, and hence  $0=\chi(h)=p^k\psi(h)$, whence (b).

(c) If $\psi$ is $G$-invariant then $\psi^g=\psi$, and hence
$\chi|_N=p^k\cdot \psi$. It follows that $\psi$ is $p$-vanishing whence the
result.
\end{proof}

\begin{rem}
Let $G,N,p,\chi,\psi$ be as in Proposition~\ref{ga1}. Then $\psi$ is not
necessarily $p$-vanishing. Indeed, let $C=\langle c \rangle$ be the cyclic
group of order~4, and let $\ep$ be a square root of $-1$. Define
$\mu_i\in\Irr(C)$ $(i=1,2,3,4)$ by $\mu_i(c)=\ep^i$. Then
$\sum_i \mu_i=\rho_C^\reg$, the regular character of $C$. Let $D$ be
the dihedral group of order~8 with normal subgroup $C$. Then
$(\sum_i \mu_i)^D=\rho_D^\reg$. One observes that $\mu_1^D=\mu_3^D$, and
hence $(2\mu_1+\mu_2+\mu_4)^D=\rho_D^\reg$. However, $2\mu_1+\mu_2+\mu_4$
is not a 2-vanishing character of $C$.
\end{rem}

\begin{cor}
 Let $G,N$ be as in Proposition~{\rm\ref{ga1}}, and let $\chi$ be a
 Steinberg-like character of $G$. Suppose that every \ir character of $N$ of
 degree at most $|N|_p$ is $G$-invariant. Then $\chi=\psi^G$ for some
 Steinberg-like character $\psi$ of $N$. In particular, if $N$ does not have
 Steinberg-like characters then neither has $G$.
\end{cor}

\begin{proof}
By Proposition \ref{ga1}(a), $\chi=\psi^G$ for some character $\psi$ of $N$.
Clearly, $\psi (1)=\chi(1)/|G:N|=|G|_p/|G:N|=|N|_p$, so, by assumption, every
\ir constituent of $\psi$ is $G$-invariant.
Therefore, so is $\psi$, and the claim follows from Proposition~\ref{ga1}(c).
\end{proof}

\begin{lem}   \label{ob4}
 Let $G$ be a finite group and $N\lhd G$ a normal subgroup of $p$-power index.
 Suppose that $l_p(\chi)\geq m$ for some integer $m>0$ and every $p$-vanishing
 character $\chi$  of $G$. Then $l_p(\chi_1)\geq m$ for every $p$-vanishing
 character $\chi_1$ of $N$.
\end{lem}

\begin{proof}
Suppose the contrary. Let $\chi_1$ be a $p$-vanishing character of $N$ such
that $l_p(\chi_1)< m$. Then the induced character $\chi_1^G$ is
$p$-vanishing and $l_p(\chi_1^G)=l_p(\chi_1)< m$. This is a contradiction.
\end{proof}

The following fact is well known.

\begin{lem}   \label{rk1}
 Let $G$ be a finite group and $N\lhd G$ a normal subgroup of $p$-power index.
 Let $F$ be an algebraically closed field of characteristic~$p$. Let $\Phi$
 be a projective indecomposable $FG$-module. Then $\Phi=\Psi^G$, where
 $\Psi$ is a projective indecomposable $FN$-module and $l_p(\Psi)=l_p(\Phi)$.
\end{lem}

\begin{proof}
It is well-known that induction sends projective modules to projective
modules. Furthermore, by Green's indecomposability theorem \cite[Thm.~3.8]{Fe}
induction from normal subgroups of $p$-power index preserves indecomposability.
So, if $\Psi$ is an indecomposable direct summand of $\Phi|_N$, then $\Psi$
is projective, $\Psi^G$ is projective indecomposable and so $\Psi^G=\Phi$.
The statement $l_p(\Psi)=l_p(\Phi)$ also follows as $|G:N|=|G:N|_p$ by
assumption.
\end{proof}

\section{Irreducible Steinberg-like characters for simple groups}\label{sec:ic1}

Here we complete the list of irreducible characters of simple groups $G$ of
degree $|G|_p$.
For this it suffices to extract the characters of degree $|G|_p$ from the list
of irreducible characters of prime-power degree obtained in \cite[Thm.~1.1]{MZ}.
This list already appeared in \cite[Prop.~2.8]{Z13}, where the case
with $p=3$, $G=\tw2F_4(2)'$ was inadvertently omitted.

Note that an irreducible character is Steinberg-like if and only
if it is $\Syl_p$-regular.

\begin{prop}   \label{prop:ir2}
 Let $G$ be a non-abelian simple group. Suppose that $G$ has an irreducible
 $\Syl_p$-regular character $\chi$. Then one of the following holds:
 \begin{enumerate}
  \item[\rm(1)] $G$ is a simple group of Lie type in characteristic $p$ and
   $\chi$ is its Steinberg character;
  \item[\rm(2)] $G=\PSL_2(q)$, $q$ even, and $p=\chi(1)=q\pm 1$, or
   $G=\SL_2(8)$, $p=3$ and $\chi(1)=9;$
  \item[\rm(3)] $G=\PSL_2(q)$, $q$ odd, $\chi(1)=(q\pm1)/2$ is a
   $p$-power for $p>2$, or $p=2$ and $\chi(1)=q\pm 1$ is a $2$-power;
  \item[\rm(4)] $G=\PSL_n(q)$, $q>2$, $n$ is an odd prime, $(n,q-1)=1$, such
   that $\chi(1)= (q^n-1)/(q-1)$ is a $p$-power;
  \item[\rm(5)] $G=\PSU_n(q)$, $n$ is an odd prime, $(n,q+1)=1$, such that
   $\chi(1)=(q^n+1)/(q+1)$ is a $p$-power;
  \item[\rm(6)] $G=\PSp_{2n}(q)$, $n>1$, $q=r^k$ with $r$ an odd prime, $kn$
   is a $2$-power such that $\chi (1)=(q^{n}+1)/2$ is a $p$-power;
  \item[\rm(7)] $G=\PSp_{2n}(3)$, $n>2$ is a prime such that
   $\chi(1)=(3^{n}-1)/2$ is a $p$-power;
  \item[\rm(8)] $G=A_{p+1}$ and $\chi(1)=p$;
  \item[\rm(9)] $G=\Sp_6(2)$ and $\chi(1)=7$;
  \item[\rm(10)] $G\in \{M_{11},M_{12}\}$ and $\chi(1)=11$;
  \item[\rm(11)] $G\in \{M_{11},\PSL_3(3)\}$ and $\chi(1)=16$;
  \item[\rm(12)] $G\in \{M_{24},Co_2,Co_3\}$ and $\chi(1)=23$;
  \item[\rm(13)] $G=\tw2F_4(2)'$ and $\chi(1)=27$;
  \item[\rm(14)] $G=\PSU_3(3)\cong G_2(2)'$ and $\chi(1)=32$; or
  \item[\rm(15)] $G=G_2(3)$ and $\chi(1)=64$.
 \end{enumerate}
\end{prop}

\medskip
The problem of determining the minimal degree of irreducible characters of
$p$-defect~0 looks much more complicated.

\begin{rem}
Let us point out the following cases not explicitly mentioned in
Proposition~\ref{prop:ir2}. $\SL_3(2)\cong \PSL_2(7),$ $A_6\cong \PSL_2(9)$,
$\PSU_4(2)\cong \PSp_4(3)$, $A_8\cong \SL_4(2)$.
\end{rem}

\section{Cyclic Sylow $p$-subgroups}   \label{sec:cycSyl}

In this section we determine the reducible Steinberg-like characters for
simple groups with cyclic Sylow $p$-subgroups.

\begin{prop}   \label{prop:cyc-TI}
 Let $G$ be a finite group with a cyclic TI Sylow $p$-subgroup $S$, and assume
 that $N_G(S)/S$ is abelian. Then $l_p(\tau)=\lfloor \tau(1)/|S|\rfloor$ for
 all $\tau\in\Irr(G)$.
\end{prop}

\begin{proof}
Let $N:=N_G(S)$. By assumption, $N/S$ is abelian of order prime to $p$, so it
has $|N:S|$ irreducible $p$-Brauer characters of degree~1. Hence, each of the
corresponding PIMs of $N$ has dimension~$|S|$. Since the Brauer tree for any
$p$-block of $N$ is a star, all PIMs are uniserial
\cite[Ch.~VII, Cor.~2.22]{Fe}. But then by \cite[Ch.~I, Thm.~16.14]{Fe}, any
indecomposable $FN$-module, where $F$ is a sufficiently large field of
characteristic~$p$, is a quotient of a PIM, so has dimension strictly smaller
than $|S|$ if it is not projective.
\par
Now let $\tau\in\Irr(G)$. If $\tau$ is of $p$-defect zero, $\tau|_S$ is a
multiple of $\rho_S^\reg$, and the claim follows. Else, $\tau$ lies in a block
of full defect, and there exists an indecomposable $FG$-module $X$ with lift
$\tau$ \cite[Ch.~I, Thm.~17.12]{Fe}. Then $X|_{FN}=Y\oplus P$,
where $P$ is projective (and hence of dimension divisible by $|S|$) and $Y$
is the Green correspondent of $X$, an indecomposable, non-projective
$FN$-module \cite[Ch.~VII, Lem.~1.5]{Fe}. Thus, $\dim Y<|S|$ by what we said
before, so $\tau(1)/|S|\le l_p(\tau)<\tau(1)/|S|+1$, and the claim follows.
\end{proof}

\begin{lem}   \label{ms1}
 Let $G$ be a non-abelian simple group. Let $p$ be a prime such that a Sylow
 $p$-subgroup of $G$ is cyclic. Let $\mu$ denote the minimal degree of any
 non-linear irreducible character of $G$. Then $2\mu>|G|_p$, except in the
 case where $G=\PSL_2(p)$, $p\equiv 3\pmod 4$ and $\mu =(p-1)/2.$
\end{lem}

\begin{proof}
The values of $\mu=\mu(G)$ for every simple group $G$ are either known
explicitly or there is a good lower bound. For the sporadic simple groups one
can inspect \cite{Atl}, for the alternating groups $A_n$ we have $\mu(A_n)=n-1$
for $n>5$, and $\mu(A_5)=3$, for simple groups $G$ of Lie type the values
$\mu(G)$ are listed in \cite{TZ96}. The lemma follows by comparison of these
data with $|G|_p$.
\end{proof}

\begin{prop}   \label{ms2}
 Let $p$ be a prime and let $G$ be a non-abelian simple group with a cyclic
 Sylow $p$-subgroup $S$. Let $\chi$ be a $\Syl_p$-regular character of $G$.
 Then one of the following holds:
 \begin{enumerate}
  \item[$(1)$] $\chi$ is irreducible of degree $|G|_p$;
  \item[$(2)$] $(\chi, 1_G)=1$, $\tau:=\chi-1_G$ is irreducible and
   $(\tau|_S,1_S)=0$; or
  \item[$(3)$] $G=\PSL_2(p)$, $p\equiv 3\pmod 4$ and $\chi=1_G+\tau_1+\tau_2$,
   where $\tau_1,\tau_2$ are distinct irreducible characters of degree
   $(p-1)/2$.
 \end{enumerate}
\end{prop}

\begin{proof}
Suppose that $\chi$ is reducible.
The result for $G=\PSL_2(p)$ easily follows by computation with the character
table of this group. Suppose $G\not\cong \PSL_2(p)$. Let $\tau\neq 1_G$ be an
irreducible constituent of $\chi$. By Lemma~\ref{ms1}, $\chi=\tau+k\cdot 1_G$,
where $k=|G|_p-\tau(1)$. Therefore, $1_G$ is a constituent of $\chi$.
By Lemma \ref{lem:a51}, $k=1$ and $(\tau|_S,1_S)=0$.
\end{proof}

\begin{prop}   \label{ms4}
 Let $p$ be a prime and let $G$ be a non-abelian simple group with a cyclic
 Sylow $p$-subgroup $S$. Then $G$ has a reducible $\Syl_p$-regular character
 $\chi$ if and only if one of the following holds:
 \begin{enumerate}
  \item[$(1)$] $G=\PSL_2(q)$, $q>4$ even, $|S|=q+1$;
  \item[$(2)$] $G=\PSL_2(p)$, $|S|=p>5$;
  \item[$(3)$] $G=\PSL_n(q)$, $n$ is an odd prime, $n{\not|}(q-1)$,
   $|S|=(q^n-1)/(q-1)$;
  \item[$(4)$] $G=\PSU_n(q)$, $n$ is an odd prime, $n{\not|}(q+1)$,
   $|S|=(q^n+1)/(q+1)$;
  \item[$(5)$] $G=A_p$, $|S|=p\ge5$;
  \item[$(6)$] $G=M_{11}$, $|S|=11$; or
  \item[$(7)$] $G=M_{23}$, $|S|=23$.
 \end{enumerate}
 Furthermore, in each case $(1)$--$(7)$, $C_G(S)=S$ and $\chi$ is
 Steinberg-like. In addition, $\chi-1_G$ is an irreducible character of $G$,
 unless possibly when $(2)$ holds, when $\chi-1_G$ may be the sum of two
 irreducible constituents of equal degree.
\end{prop}

\begin{proof}
The additional statement follows from Proposition~\ref{ms2}.
If $\chi-1_G$ is reducible, we have the case (3) of Proposition~\ref{ms2}.
So we may assume that $\tau=\chi-1_G$ is irreducible and thus that
$(\tau|_S,1_S)=0$. The irreducible characters of $G$ of level 0 are
determined in \cite[Thm.~1.1]{Z99}, so $\tau$ belongs to the list in
\cite[Thm.~1.1]{Z99}. If we drop from that list the characters of degree other
than $|S|-1$, the remaining cases are given in the statement of the proposition.
(Note that the list in \cite[Thm.~1.1]{Z99} includes quasi-simple groups
so one first needs to delete the representations non-trivial on the center.
For instance, if $G=\PSp_{2n}(q)$ then  $|S|=(q^n+1)/2$ is odd, and hence
$\tau(1)=\chi(1)-1=|S|-1$ is even. However, every irreducible representation
of $\Sp_{2n}(q)$ of even degree $(q^n-1)/2$ is faithful.  In other words,
$G$ has no irreducible representation of even degree $(q^n-1)/2$. In contrast,
there do exist irreducible representations of $G=\PSL_n(q)$ and $\PSU_n(q)$
for $n$ odd of degree $|S|-1$.)

To prove the converse, we have to show that in each case $1_G+\tau$ is
$\Syl_p$-regular, that is, $\chi|_S=\rho_S^\reg$. Let $\Psi$ be a
representations of $G$ afforded by $\tau$. Let $s\in S$ with $S=\lan s\ran$.
By \cite[Cor.~1.3(2)]{Z99}, the multiplicity of every eigenvalue of $\Psi(s)$
is 1. As $\det \Psi(s)=1$, it follows that $1$ is not an
eigenvalue of $\Psi(s)$. Therefore, $\chi|_S=\rho_S^\reg$, as required.

Next, we show that $C_G(S)=S$. In cases (6) and~(7) this follows by inspection
in \cite{Atl}. The cases (1), (2), (5) are trivial. In cases (3), (4) one can
take the preimage $T$, say, of $S$ in $G_1=\SL_n(q)$, $\SU_n(q)$, respectively.
Then $T$ is irreducible on the natural module for $G_1$. The groups
$C_{G_1}(T)$ are described by Huppert \cite[S\"atze~4,5]{Hu}. It easily
follows that $T$ is self-centralising in $G_1$. Then $C_G(S)=S$ unless
$[g,T]\subseteq Z(G_1)$ for some
$g\in N_{G_1}(T)\setminus T$. By order consideration, $S$ is a Sylow
$p$-subgroup of $G$, so $g$ is not a $p$-element. Let $t\in T$. Then
$[g,t^i]=[g^i,t]=1$ for $i=|S|$, so $g^{|S|}\in C_{G_1}(T)=T$ by the above.
This is a contradiction as $S$ is a Sylow $p$-subgroup.

It follows that every element of $G$ is either a $p$- or a $p'$-element.
Therefore, $\chi$ is Steinberg-like if and only if $\chi|_S=\rho^\reg_S$.
\end{proof}

\begin{lem}   \label{re56}
 Under the assumptions and in the notation of Proposition $\ref{ms4}$ we have:
 \begin{enumerate}
 \item[\rm(a)] $\chi$ is unique unless $(2)$ or $(6)$ holds;
 \item[\rm(b)] $\chi$ is the character of a projective module when $(1)$, $(3)$,
  $(5)$ or $(7)$ holds; and
 \item[\rm(c)] $\chi-1_G$ is a proper character, and if $m$ is the minimal
  degree of a non-linear character of $G$ then either $m=\chi(1)-1$, or $(1)$
  holds and $m=\chi(1)-2$, or $(2)$ holds and $m=(\chi(1)\pm1)/2$.
 \end{enumerate}
\end{lem}

\begin{proof}
(a) Let $\tau=\chi-1_G$. Then $\tau(1)=|S|-1$ and $\tau$ is irreducible unless
(2) holds. We show that an irreducible character of this degree is unique
unless (2) or~(6) holds.
If $G=M_{23}$, this follows from the character table of this group, for
$A_p$ this is well known. For $G=\PSL_n(q)$, $n>2$, and $\PSU_n(q)$, $n>2$,
this is observed in \cite[Table II]{TZ96}. \par
In case (2) the number of characters equals the number of irreducible characters
of degree $p-1$, which is $(p-3)/4$ if $p\equiv3\pmod4$, otherwise $(p-1)/4$.
If $G=M_{11}$ then there are three Steinberg-like characters, see \cite{Atl}.

(b) Recall that the principal projective indecomposable module is the only
PIM whose character contains $1_G$ as a constituent. All the characters $\chi$
in Proposition~\ref{ms4} contain $1_G$ as a constituent. Therefore,
if $\chi$ is the character of a projective module $\Phi$, say, then $\Phi$ is
indecomposable and principal. So we compare the list of characters $\chi$
in Proposition~\ref{ms4} with the main result of \cite{MW}. The comparison
rules out the case (4) of Proposition~\ref{ms4}.
Furthermore, if $G$ admits at least two Steinberg-like characters then
at most one of them can be the character of a projective module.
By~(a) this leaves us with cases (1), (5) and~(7). As in each of these cases
$\chi$ is unique, it must be the character of the principal projective
indecomposable module listed in \cite{MW}.

(c) This follows by inspection in \cite[Table II]{TZ96}.
\end{proof}

\begin{rem}
 The group $G=\PSL_2(p)$ has several $\Syl_p$-regular characters, all of them
 are Steinberg-like, and only one of them is projective.
\end{rem}

\section{Sporadic groups}   \label{sec:spor}

\begin{thm}   \label{thm:spor}
 Let $G$ be a sporadic simple group. Then $G$ does not have a reducible
 $\Syl_p$-regular character unless one of the following holds:
 \begin{enumerate}
  \item[\rm(1)] $G=M_{12}$, $p=3$, four characters with constituents of degrees
   $11$ and~$16$ each, all Steinberg-like;
  \item[\rm(2)] $G=M_{24}$, $p=2$, six characters, none of them Steinberg-like;
  \item[\rm(3)] $G=M_{11}$, $|S|=11$; or
  \item[\rm(4)] $G=M_{23}$, $|S|=23$.
 \end{enumerate}
\end{thm}

\begin{proof}
For most groups and primes, by \cite{Atl} there is a conjugacy class of
non-trivial $p$-elements taking strictly positive value at all irreducible
characters of degree at most $|G|_p$. In a few cases, like in $Co_3$ and
$Fi_{23}$ at $p=3$, or $Co_1$ and $J_4$ at $p=2$, one has to solve a little
linear system of equations for non-negative integral solutions.
The only cases where such solutions exist are listed in the statement.
Note that the cases (3), (4) occur also in Proposition~\ref{ms4}.
\end{proof}

\section{Alternating groups}   \label{sec:alt}
In this section we consider Steinberg-like characters of alternating groups.

\subsection{Alternating groups for $p>2$}

For odd primes we give a short proof using a recent result of Giannelli and
Law \cite{GL17} which replaces our earlier more direct proof.

\begin{lem}   \label{pp5}
 Let $G=A_p$, $p>3$, and $\chi\in\Irr(G)$. Then
 $l_p(\chi)=\lfloor\chi(1)/p\rfloor$. In addition, $l_p(\chi)\neq 1$ for
 $p>7$ (this fact has also been observed in \cite{RZ}).
\end{lem}

\begin{proof}
The first part is just Proposition~\ref{prop:cyc-TI}. In addition, if $p>7$
then $G$ has no irreducible character of degree $d$ for $p\le d<2p$.
This implies the claim.
\end{proof}

\begin{lem}   \label{ii6}
 Let $n=kp$, where $p>5$ and $k<p$. Let $G=A_n$ and let $\chi$ be a
 $p$-vanishing character. Then $l_p(\chi)\geq 2^{k-1}$, equivalently,
 $\chi(1)\geq 2^{k-1}|G|_p$.
\end{lem}

\begin{proof}
For $k=1$ the lemma is trivial. Let $k>1$. Let $X_1\cong A_p$,
$X_2\cong A_{n-p}$ be commuting subgroups of $G$. Set $X=X_1X_2$.
Then $\chi|_X=\sum\eta_i\si_i$, where the $\si_i$'s are $p$-vanishing
characters of $X_2$ and the $\eta_i$'s are distinct irreducible characters
of $X_1$ (Lemma \ref{nn2}). By induction, $l_p(\si_i)\geq 2^{k-2}$.
If $l_p(\eta_i)\geq 2$ for some $i$ then
$l_p(\chi)\geq l_p(\eta_i\si_i)\geq 2^{k-1}$. By Lemma \ref{pp5}, if
$l_p(\eta_i)<2$ and $p>7$ then $l_p(\eta_i)=0$, and hence either
$\eta_i=1_{X_1}$ or $\eta_i$ is the unique irreducible character of degree
$p-1$. (If $p=7$ then we may have $\eta_i(1)=10$, see \cite{RZ}.)

Suppose the lemma is false and $p>7$. Then we can rearrange the above to get
$$\chi|_X=1_{X_1}\cdot \si_1+\eta_2\cdot \si_2,$$
where $\eta_2(1)=p-1$ and
$\si_1,\si_2$ are $p$-vanishing characters of $X_2$. It follows that
$\chi|_{X_1}$, as well as $\tau|_{X_1}$ for every irreducible constituent
$\tau$ of $\chi$, contains no irreducible constituent distinct from
$1_{X_1},\eta_2$. It is well known and easily follows from the branching
rule that this implies $\tau(1)=n-1$ or $1$. Recall that $G$ has a single
character of degree $n-1$. Therefore, $\chi=a\cdot 1_G+b\cdot \tau$, where
$\tau(1)=n-1$. Let $x\in X_1$ be of order $p$. Then $\tau(x)=
n-p-1>0$, which implies $\chi(x)>0$.

Suppose $p=7$. Then $\eta_i(1)\in\{1,6, 10\}$. There are two irreducible
characters of $X_1$ of degree 10, let us denote them by $\eta_3,\eta_3'$.
Therefore, assuming the lemma is false, we can write
$\chi|_X=1_X\cdot \si_1+\eta_2\si_2+\eta_3\si_3+\eta_3'\si_4$.
Let $1\neq x\in X_1$ be a $p$-element. Then $\eta_3(x)=\eps+\eps^4+\eps^2$ and
$\eta'_3(x)=\eps\up+\eps^{-4}+\eps^{-2}$, where $\eps$ is some primitive 7th
root of unity. As $\chi(x),\eta_2(x)$ are integers, so is
$\eta_3(x)\si_3(1)+\eta_3'(x)\si_4(1)$. This implies $\si_3(1)=\si_4(1)$.
Then $\chi(1)=\si_1(1)+(p-1)\si_2(1)+20\si_3(1)>14\si_3(1)$, and the lemma
follows unless $\si_3(1)=0$.
If $\si_3(1)=0$ then $\chi|_X=1_{X_1}\cdot \si_1+\eta_2\cdot \si_2$, and the
above argument applies.
\end{proof}

\begin{lem}   \label{lem:p-cycle}
 Let $p\ge3$ be odd and let $\lam$ be a hook partition of $n\ge 2p$. Then the
 corresponding character $\chi^\lam$ of $S_n$ takes a positive value on
 $p$-cycles.
\end{lem}

\begin{proof}
It is well known that any hook character $\chi^\lam$ is the $m$th exterior
power, for some $0\le m\le n-1$, of the irreducible reflection character
$\rho_n$ of $S_n$ (the constituent of degree~$n-1$ of the natural permutation
character~$\pi_n$). Let $Y=Y_1\times Y_2$, with $Y_1=S_p$ and $Y_2=S_{n-p}$, be
a Young subgroup of $S_n$ and $g=g'\times1\in Y$ a $p$-cycle. Clearly
$\rho_n|_Y=\rho_p\boxtimes 1_{Y_2}+1_{Y_1}\boxtimes(\pi_{n-p})$, and
$\Lambda^i(\rho_p)(g')=(-1)^i$ for $i<p$, $\Lambda^i(\rho_p)(g')=0$ for
$i\ge p$. Thus
$$\chi^\lam(g)=\Lambda^m(\rho_n)(g)
  =\sum_{i=0}^m\Lambda^{i}(\rho_p)(g')\Lambda^{m-i}(\pi_{n-p})(1)
  =\sum_{i=0}^{\min(p-1,m)}(-1)^{i}\binom{n-p}{m-i},
$$
which clearly is positive for $m\le (n-p)/2$ since the binomial coefficients
are (strictly) increasing up to the middle.
Now observe that it suffices to prove the claim for $n=2p$, since the
restriction of a hook character from $S_n$ to $S_{n-1}$ only contains hook
characters. But for $n=2p$ we are done since by symmetry we may assume that
$m\le p=(n-p)/2$.
\end{proof}

\begin{thm}   \label{thm:An}
 Let $p$ be odd and $G=A_n$ with $n>\max\{6,p+1\}$. Then $G$ has no
 $\Syl_p$-regular character. If $n=p+1>4$ then every $\Syl_p$-regular character
 of $G$ is irreducible, unless $(n,p)=(6,3)$.
\end{thm}

\begin{proof}
If $p\le n<2p$ the Sylow $p$-subgroups of $G$ are cyclic and so the claim is
in Proposition~\ref{ms4}. Now assume that $n\ge2p$ and let $S$ be a Sylow
$p$-subgroup of $S_n$.
First assume that $n\ne p^k$ for some $k\ge2$ and that $n>10$ when $p=3$. Then
by the main result of \cite{GL17}, the restriction of any irreducible
character of $S_n$ to $S$ contains the trivial character. A moments thought
shows that the same is true for the restriction of any irreducible character
of $A_n$ to $S$. So by Lemma~\ref{lem:a51} any $\Syl_p$-regular character of
$A_n$ is irreducible. \par
Now assume that $n=p^k$ for some $k\ge2$, and $n>10$ when $p=3$. Then again by
\cite[Thm.~A]{GL17} the only irreducible characters of $S_n$ whose restriction
to $S$ does not contain the trivial character are the two characters of degree
$n-1$. So the only irreducible character of $A_n$ whose restriction to $S$
does not contain the trivial character is $\psi$ of degree $n-1$. Hence a
$\Syl_p$-regular character $\chi$ of $A_n$ has the form $\chi=a\psi+\psi'$,
for some $a\ge0$ and some $\psi'\in\Irr(A_n)$. Let $g\in A_n$ be a $p^k$-cycle.
Then $\psi(g)=-1$, and by the Murnaghan--Nakayama rule any irreducible
character of $S_n$ takes value $0$ or $\pm1$ on $g$. In particular, if $\chi$
is reducible then we have that $a=1$ and $\psi'(g)\ne0$. But then $\psi'$ is
parametrised by a hook partition, of degree $\binom{n-1}{m}$ for some $m\le n$.
But then $\chi$ takes positive values on $p$-cycles by Lemma~\ref{lem:p-cycle},
a contradiction.
\par
Finally, the cases when $p=3$ and $6\le n\le 10$ can easily be checked
individually. For example, all irreducible characters of $A_9$ of degree at
most~81 are non-negative on class 3C, and those which vanish there are positive
either on class 3B or 3A. So $A_9$ has no $\Syl_3$-regular character. As
$A_{10}$ has the same Sylow 2-subgroup, this also deals with $n=10$.
\end{proof}

\begin{cor}   \label{p11}
 Let $G$ be a finite group and $p>2$. Suppose that $G$ has a subgroup
 $P$ containing a Sylow $p$-subgroup of $G$ and such that $P/O_p(P)\cong A_n$
 with $n>\max\{6,p+1\}$. Then $G$ has no Steinberg-like character.
\end{cor}

\begin{proof}
This follows from Lemma~\ref{hc1} and Theorem~\ref{thm:An}.
\end{proof}

\subsection{Alternating groups for $p=2$}   \label{subsec:alt p=2}
The situation is more complicated in the case of $p=2$ and we don't have
complete results. This is in part due to the existence of an infinite family
of examples which we now construct.

Set $\Gamma=\sum_{i=1}^n \Gamma_i$, where $\Gamma_i$ is the \ir character of
$S_n$ corresponding to the partition $[i,1^{n-i}]$ for $i>1$, and $[1^n]$ for
$i=1$. So the Young diagram $\gamma_i$ of $\Gamma_i$ is a hook with leg length
$n-i$, and
$$\Gamma_i(1)=\frac{n!}{n(n-i)!(i-1)!}=\binom{n-1}{i-1}$$
so $\Gamma(1)=\sum_{i=1}^n \Gamma_i(1)=2^{n-1}$.

\begin{lem}   \label{h66}
 Let $0<m<n$, where $m$ is even, and $g=ch\in S_m\times S_{n-m}\le S_n$,
 where $c$ is an $m$-cycle  and $h$ fixes all letters moved by $c$. Let
 $\Gamma ^{n-m}_k\in\Irr(S_{n-m})$ correspond to the hook partition
 $[k,1^{n-m-k}]$. Then
 $\Gamma_i(g)=\begin{cases}  -\Gamma ^{n-m}_i(h)& \text{if }i\leq m,\cr
  \Gamma ^{n-m}_{i-m}(h)& \text{if }n-m<i,\cr
  \Gamma ^{n-m}_{i-m}(h)-\Gamma ^{n-m}_{i}(h)& \text{if } m<i\leq n-m.
 \end{cases}$
\end{lem}

\begin{proof}
One observes that the restriction of $\Gamma_i$ to $S_m\times S_{n-m}$ is a sum
of \ir characters $\si \tau$, where $\si$, $\tau$ are \ir characters of $S_m$,
$S_{n-m}$, resp., and both $\si,\tau$ are hook characters of the respective
groups (see \cite[Lemma 21.3]{Ja}). Next we use \cite[Lemma 21.1]{Ja} which
states that $\Gamma_i(g)=\sum_j (-1)^j \chi_{\nu_j}(h)$, where
$\chi_{\nu_j}\in\Irr(S_{n-m})$,
${\nu_j}$ is the Young diagram of $\chi_{\nu_j}$, and ${\nu_j}$ is such that
$\gamma_i\setminus {\nu_j}$ is a skew hook with leg length $j$. In our case
$\gamma_i$ is a hook, so the rim of $\gamma_i$ is $\gamma_i$ itself. By
definition, a skew hook is connected,  so it is either a row or a column in our
case, and hence $j=0$ or $j=m-1$. (A column hook of length $m$ has leg length
$m-1$, which is odd as $m$ is even.) If $j=0,m-1$ then
$\nu=[i-m,1^{n-i}],[i,1^{n-i-m}]$, respectively.
This is a proper diagram if and only if $i>m$, resp., $n-i\geq m$. So if
$n-i<m$ then $\nu=[i-m,1^{n-i}]$, $j=0$, and
$\Gamma_i(g)=\chi_{\nu}(h)=\Gamma^{n-m}_{i-m}(h)$; if $i\leq m$ then
$\nu=[i,1^{n-i-m}]$, $j=m-1$, and
$\Gamma_i(g)=-\chi_{\nu}(h)=-\Gamma^{n-m}_{i}(h)$; if $m<i\leq n-m $ then
$\nu\in \{ [i-m,1^{n-i}],[i,1^{n-i-m}]\}$
and  $\Gamma_i(g)=\Gamma^{n-m}_{i-m}(h)-\Gamma^{n-m}_{i}(h)$, as claimed.
\end{proof}

\begin{prop}   \label{so1}
 Suppose that $n$ is even. Then:
 \begin{enumerate}
  \item[\rm(a)] $\Gamma$ is a $2$-vanishing character of $S_n$.
  \item[\rm(b)] $\Gamma$ is Steinberg-like if and only if $n=2^k$ for some
   integer $k>0$.
 \end{enumerate}
\end{prop}

\begin{proof}
(a) Let $g\in S_n$ be of even order. Suppose first that $g$ is a cycle of
length $n$. By \cite[Lemma 21.1]{Ja}, $\Gamma_i(g)=(-1)^{n-i}$, so
$\Gamma(g)=0$.

Suppose that  $g$ is not a cycle of length $n$. Then we can express $g$ as
the product $ch$ of a cycle $c$ of even size $m$, say, and an element $h$
fixing all letters moved by $c$. Then $g\in S_m\times S_{n-m}$.
By Lemma \ref{h66}, we have
$$\Gamma(g)=\sum_{i=1}^n \Gamma_i(g)
  =\sum_{i=m+1}^n \Gamma^{n-m}_{i-m}(h)-\sum_{i=1}^{n-m}\Gamma^{n-m}_i(h)
  =\sum_{k=1}^{n-m} \Gamma^{n-m}_{k}(h)-\sum_{i=1}^{n-m}\Gamma^{n-m}_i(h)=0.$$

\medskip
(b) If $n=2^k$ then $|S_n|_2=2\cdot |S_{n/2}|^2_2$. As $|S_2|_2=2$, by
induction we have
$$|S_n|_2=2\cdot (2^{2^{k-1}-1})^2=2^{2^{k}-1}=2^{n-1}.$$

Write $n=2^k+l$ where $0<l<2^k$. Then $|S_n|_2=|S_{2^k}|_2\cdot|S_l|_2$.
By induction, $|S_l|_2\leq 2^{l-1}$, so
$|S_n|_2=2^{2^k-1}\cdot |S_l|_2\leq 2^{(2^k-1)+(l-1)}=2^{2^k+l-2}=2^{n-2}$.
The statement follows as $\Gamma(1)=2^{n-1}$.
\end{proof}

\begin{cor}   \label{so2}
 Let $n$ be even, and $\Gamma^0=\sum_{i=1}^{n/2}\Gamma_i|_{A_n}$. Then
 $\Gamma^0$ is a $2$-vanishing character of $A_n$. If $n=2^k$ then this
 character is Steinberg-like.
\end{cor}

\begin{proof}
The characters $\Gamma_i$ remain \ir under restriction to $A_n$ and
$\Gamma_i|_{A_n}=\Gamma_{n-i+1}|_{A_n}$. It follows that
$\Gamma|_{A_n}=2\Gamma^0$. Therefore, $\Gamma^0(g)=\Gamma(g)/2=0$ by
Proposition~\ref{so1} for elements $g$ of even order. The last claim follows
from Proposition~\ref{so1}(b).
\end{proof}

Suppose that $n$ is odd. Set
$$\begin{aligned}
  \Gamma^e&=\sum _{i=1}^{(n-1)/2}\Gamma_{2i}
  =\Gamma_2+\Gamma_4+\cdots+\Gamma_{n-1},\quad\text{and}\\
  \Gamma^o&=\sum_{i=1}^{(n+1)/2}\Gamma_{2i-1}
  =\Gamma_1+\Gamma_3+\cdots +\Gamma_n.
\end{aligned}$$
Observe that $\Gamma_i|_{S_{n-1}}=\Gamma^{n-1}_{i}+\Gamma^{n-1}_{i-1}$ provided
$1<i<n$, and $\Gamma_{1}|_{S_{n-1}}=\Gamma^{n-1}_{1}$, $\Gamma_{n}|_{S_{n-1}}
=\Gamma^{n-1}_{n-1}$. Therefore,
$\Gamma^e|_{S_{n-1}}=\Gamma^{n-1}_{1}+\cdots +\Gamma^{n-1}_{n-1}=
\Gamma^o|_{S_{n-1}}$. As $\Gamma=\Gamma^e+\Gamma^o$ we have
$\Gamma^e(1)=\Gamma^o(1)=\Gamma(1)/2=2^{n-2}$.

\begin{prop}   \label{so3}
 Suppose that $n$ is odd. Then:
 \begin{enumerate}
  \item[\rm(a)] $\Gamma^e$ and $\Gamma^o$ are $2$-vanishing characters of $S_n$.
  \item[\rm(b)] $\Gamma^e$ is Steinberg-like if and only if $n=2^k+1$ for some
   integer $k>0$.
 \end{enumerate}
\end{prop}

\begin{proof}
(a) Let $g\in S_n$ be of even order, and $g=ch$ where $c$ is a cycle of even
size $m$. By Lemma~\ref{h66},
$$\Gamma^e(g)=\sum_{i=1}^{(n-1)/2}\Gamma_{n-2i+1}(g)
  =\sum_{i=1}^{(n-1-m)/2}\Gamma^{n-m}_{n-m-2i+1}(h)
   -\sum_{i=(m+2)/2}^{(n-1)/2}\Gamma^{n-m}_{n-2i+1}(h),$$
as $\gamma^{n-m}_{n-m-2i+1}$ is a proper diagram only for $i<(n-m)/2$ and
$\gamma^{n-m}_{n-2i+1}$ is a proper diagram only for $i\geq(m+2)/2$. Set
$k=i-m/2$. So the second sum can be written as
$\sum_{k=1}^{(n-1-m)/2}\Gamma^{n-m}_{n-m-2k+1}(h)$, whence $\Gamma^e(g)=0$.

Similarly,
$$\Gamma^o(g)=\sum_{i=1}^{(n+1)/2}\Gamma _{n-2i+2}(g)
  =\sum_{i=1}^{(n-1-m)/2}\Gamma^{n-m}_{n-m-2i+2}(h)
   -\sum_{i=(m+2)/2}^{(n-1)/2}\Gamma^{n-m}_{n-2i+2}(h), $$
as $\gamma^{n-m}_{n-m-2i+1}$ is a proper diagram only for $i\le(n-m-1)/2$ and
$\gamma^{n-m}_{n-2i+2}$ is a proper diagram only for $i\geq (m+2)/2$. Set
$k=i-m/2$. Then the second sum can be written as
$\sum_{k=1}^{(n-1-m)/2}\Gamma^{n-m}_{n-m-2k+2}(h)$. So $\Gamma^o(g)=0$ as well.

\medskip
(b) If $n=2^k+1$ then $|S_n|_2=|S_{n-1}|_2=2^{n-2}$ (see the proof of
Proposition~\ref{so1}(b)). By the above, $\Gamma^e(1)=\Gamma^o(1)=2^{n-2}$,
so both $\Gamma^e$ and $\Gamma^o$ are Steinberg-like. If $n-1$ is not a
2-power then $|S_n|_2=|S_{n-1}|_2<2^{n-2}$ by Proposition~\ref{so1}(b).
\end{proof}

Let $n$ be odd. Then $\Gamma_i|_{A_n}=\Gamma_{n+1-i}|_{A_n}$ is irreducible
for $i\ne(n+1)/2$, whereas
$\Gamma_{(n+1)/2}|_{A_n}$ is the sum of two irreducible constituents which we
denote by $\Gamma_{(n+1)/2}^+$ and $\Gamma_{(n+1)/2}^-$. If $n=4l+1$ then set
$$\begin{aligned}
  \Gamma^{ea}&=\sum_{i=1}^{(n-1)/4}\Gamma_{2i}|_{A_n}
    =(\Gamma_2 +\Gamma_4 +\cdots +\Gamma_{(n-1)/2})|_{A_n},\quad\text{and}\\
  \Gamma^{o\pm}
  &=\Gamma_{(n+1)/2}^\pm+\sum_{i=1}^{(n-1)/4}\Gamma_{2i-1}|_{A_n}
    =(\Gamma_1+\Gamma_3+\cdots +\Gamma_{(n-3)/2})|_{A_n}+\Gamma_{(n+1)/2}^\pm,
\end{aligned}$$
while for $n=4l+3$ we set
$$\begin{aligned}
  \Gamma^{oa}&=\sum_{i=1}^{(n+1)/4}\Gamma_{2i-1}|_{A_n}
    =(\Gamma_1+\Gamma_3+\cdots +\Gamma_{(n-1)/2})|_{A_n},\quad\text{and}\\
  \Gamma^{e\pm}
  &=\Gamma_{(n+1)/2}^\pm+\sum_{i=1}^{(n-3)/4}\Gamma_{2i}|_{A_n}
   =(\Gamma_2+\Gamma_4+\cdots+\Gamma_{(n-3)/2})|_{A_n}+\Gamma_{(n+1)/2}^\pm.
\end{aligned}$$

\begin{cor}   \label{so39}
 \begin{enumerate}
  \item[\rm(a)] Let $n=4l+1$. Then $\Gamma^{ea}$, $\Gamma^{o+}$ and
   $\Gamma^{o+}$ are $2$-vanishing characters of $A_n$.
   If $n=2^k+1$ then they are Steinberg-like characters.
  \item[\rm(b)] Let $n=4l+3>3$. Then $\Gamma^{e+}$, $\Gamma^{e-}$ and
   $\Gamma^{oa}$ are $2$-vanishing characters of $A_n$. None of them is
   Steinberg-like.
 \end{enumerate}
\end{cor}

\begin{proof}
Let $g\in A_n$ be of even order.

(a) Let $i\neq (n+1)/2$. Then $\Gamma_i$ remains \ir under restriction to
$A_n$. As $\Gamma_i$ and $\Gamma_{n-i+1}$ coincide under restriction to $A_n$,
it follows that $\Gamma^{e}|_{A_n}=2\Gamma^{ea}$, and hence $\Gamma^{ea}$ is a
2-vanishing character. As $\Gamma^{ea}(1)=2^{n-3}=|A_n|_2$,
this is Steinberg-like for $n=2^k+1$.

Observe that $\Gamma_{(n+1)/2}^+(g)=\Gamma_{(n+1)/2}^-(g)$ and
$\Gamma_{(n+1)/2}^+(g)+\Gamma_{(n+1)/2}^-(g)=\Gamma_{(n+1)/2}(g)$. It follows
that $\Gamma^{o+}(g)=\Gamma^{o-}(g)=\Gamma^{o}(g)/2$, and thus
$\Gamma^{o+}(g)=\Gamma^{o-}(g)=0$ by Proposition~\ref{so3}.
Therefore, $\Gamma^{o+}$ and $\Gamma^{o+}$ are 2-vanishing characters of $A_n$.
In addition, suppose that $n=2^k+1$. Then
$\Gamma^{o+}(1)=\Gamma^{o-}(1)=\Gamma^{o}(1)/2=|A_n|_2$, so both
$\Gamma^{o+}$ and $\Gamma^{o-}$ are Steinberg-like.

(b) Let $i\neq (n+1)/2$. Then as above it follows that
$\Gamma^{o}|_{A_n}=2\Gamma^{oa}$, and hence $\Gamma^{oa}$ is a
2-vanishing character. In addition, $\Gamma^{oa}(1)=\Gamma^{o}(1)/2=2^{n-3}$.
As here we never have $n=2^k+1$, $\Gamma^{oa}$ is not Steinberg-like.

Consider $\Gamma^{e\pm}$. Observe that
$\Gamma_{(n+1)/2}^+(g)=\Gamma_{(n+1)/2}^-(g)$ and
$\Gamma_{(n+1)/2}^+(g)+\Gamma_{(n+1)/2}^-(g)=\Gamma_{(n+1)/2}(g)$. It follows
that $\Gamma^{e+}(g)=\Gamma^{e-}(g)=\Gamma^{e}(g)/2$, and so
$\Gamma^{e+}(g)=\Gamma^{e-}(g)=0$ by Proposition~\ref{so3}.
Therefore, $\Gamma^{e+}$ and $\Gamma^{e-}$ are 2-vanishing characters of $A_n$
but not Steinberg-like.
\end{proof}

\begin{lem}   \label{lem:Sn small p=2}
 Let $2\le n\le12$. Then in addition to the character $\Gamma$ when $n=2^k$,
 and the characters $\Gamma^e$ and $\Gamma^o$ when $n=2^k+1$, the only
 Steinberg-like characters of $S_n$ are:
 \begin{enumerate}
  \item[\rm(a)] if $n=4$ the sum of all non-linear irreducible characters;
  \item[\rm(b)] if $n=6$ the irreducible character of degree~$16$;
  \item[\rm(c)] if $n=8$ the sum of the two irreducible characters of degree~$64$.
 \end{enumerate}
\end{lem}

\begin{proof}
For $n\le6$ this is easily checked from the know character tables. For $n=8$
we use a computer program to go through all possibilities. For $S_{10}$ one
checks that no character exists with the right restriction to $S_8\times S_2$,
and similarly for $S_{12}$ one considers the restriction to $S_8\times S_4$.
Finally, the cases $n\in\{7,9,11\}$ are treated by restricting to $S_{n-1}$.
\end{proof}

\begin{thm}   \label{thm:An p=2 ind}
 Suppose that the only Steinberg-like character of $S_{2^k}$, $k\ge4$, for
 $p=2$ is the one constructed in Proposition~{\rm \ref{so1}}. Then $A_n$ does not
 have Steinberg-like characters for $p=2$ for $n\ge13$ unless $n$ or $n-1$ is
 a $2$-power. In the latter case, the only Steinberg-like characters are those
 listed in Proposition~\ref{so3}.
\end{thm}

\begin{proof}
Let $\psi$ be a Steinberg-like character for $p=2$ of $A_n$, with $n\ge10$.
Then $\chi:=\psi^{S_n}$ is Steinberg-like for $S_n$. We argue by induction on
$n$ that $S_n$ does not have a Steinberg like character, unless $n$ or $n-1$
is a power of~2.  \par
Assume that $n$ is not a power of~2 and write $n=2^{a_1}+\ldots+2^{a_r}$ for
distinct exponents $a_1,\ldots,a_r>0$. By Lemma~\ref{lem:Sn small p=2} we
may assume $n\ne12$, so one of the summands, say $2^{a_1}$ is different from~4
and~8. Then the Young subgroup
$Y=Y_1\times Y_2:=S_{2^{a_1}}\times S_{n-2^{a_1}}$ of $S_n$ contains a Sylow
2-subgroup, so $\chi|_Y$ is Steinberg-like. Then by Lemma~\ref{nn2} we have
that $\chi|_Y=\sum_i \eta_i\si_i$ where $\eta_i\in\Irr(Y_1)$ are all
distinct, the $\si_i$ are $2$-vanishing characters of $Y_2$, and
$\chi_1:=\sum_i l_p(\si_i)\eta_i$ is a 2-vanishing character of $Y_1$ with
$l_2(\chi_1)=l_2(\chi)=1$. Thus by assumption $\chi_1$ is the character
$\Gamma$ from Proposition~\ref{so1}. In particular, $\chi_1$ is
multiplicity-free and so $l_p(\si_i)=1$ for all $i$. So the $\si_i$ are
Steinberg-like as well. This is not possible, unless $n-2^{a_1}$ is a
2-power as well.
\par
In the latter case, by Lemma~\ref{lem:Sn small p=2} we conclude that $n\ge17$.
The above argument shows that $\chi|_Y=\Gamma^{(1)}\boxtimes\Gamma^{(2)}$,
with $\Gamma^{(j)}$ a Steinberg-like character of $Y_j$. So in particular
$\chi|_Y$ and hence also $\chi$ is multiplicity-free. By possibly interchanging
$a_1,a_2$ we may assume that $2^{a_1}>8$. Now consider
$\chi|_{Y_1}=|Y_2|_2 \Gamma^{(1)}$, a sum of hooks. By the branching rule,
any non-hook character of $S_m$ restricted to $S_{m-1}$ contains a non-hook
character (except when $m=4$ which is excluded here). Thus, inductively,
$\chi$ cannot contain any non-hook constituent. This in turn means that all
constituents of $\chi|_{Y_2}$ are hooks and thus by induction that
$\chi|_{Y_2}=|Y_1|_2 \Gamma^{(2)}$. Now observe that by the
Littlewood--Richardson rule \cite[Lemma~21.3]{Ja}, $\Gamma_i|_Y$ contains
$\Gamma_j^{(1)}\boxtimes\Gamma_l^{(2)}$ if and only if $j+l\in\{i,i+1\}$.
Thus, on the one hand side, $\Gamma_i|_Y$ and $\Gamma_{i+1}|_Y$ have a
common constituent, and so at most every second hook character occurs in
$\chi$. On the other hand, every second hook must indeed occur. Thus either
$\chi=\Gamma^e$ or $\chi=\Gamma^o$ as defined above. If $n=2^{a_1}+1$ then
our claim follows from Proposition~\ref{so3}, otherwise the degree of
$\chi$ is larger than $|S_n|_2$.
\end{proof}

\subsection{Projective characters for $p=2$}   \label{subsec:alt proj}

\begin{lem}   \label{as7}
 Let $p=2$. Then $A_n$ has a projective character of degree $|A_n|_2$
 if and only if $S_n$ has a projective character of degree $|S_n|_2$.
\end{lem}

\begin{proof}
 This follows from Lemma \ref{rk1}.
\end{proof}

\begin{thm}   \label{22a}
 Let $p=2$ and $G=A_n$ or $S_n$ for $n>4$. Then $G$ has no reducible projective
 character of degree $|G|_2$.
\end{thm}

\begin{proof}
One can inspect the decomposition matrix modulo~2 of $G=A_n$ for $n\leq 9$ to
observe that $G$ has no projective character of degree $|G|_2$. Analysing the
character table of $G=A_n$ for $9<n\le15$ one observes that $G$ has no
$\Syl_2$-regular characters, and hence no PIM of dimension $|G|_2$.

One can inspect the decomposition matrix of $G=A_9$ to observe that the minimal
dimension of a PIM is $320$. Analysing the character table of $G=A_n$ for
$9<n\le15$ one observes that $G$ has no $\Syl_2$-regular characters, and
hence no PIM of dimension $|G|_2$.

Let $n=16$. Using the known character table of $A_{16}$ one finds that there
is a unique $\Syl_2$-regular character, viz.~the character $\Gamma^0$; it is
multiplicity-free with constituents of degrees
$1,15,105,455,1365,3003,5005,6435$, and Steinberg-like.
Recall that the principal PIM is the only one that has $1_G$ as a constituent.
However, by \cite{MW}, the principal PIM is not of degree $|G|_2$.

Let $n=2^k$, where $k>4$. Then $G$ has a subgroup $Y$ such that
$Y/N\cong A_{16}$ for some normal $2$-subgroup $N$ and $|Y|_2=|G|_2$.
Indeed, let $P_1,\ldots,P_{16}$ be a partition of $\{1,\ldots,n\}$ with all
parts of size $n/16$. If $G=S_n$ then $N$ is the direct product of 16 copies
of a
Sylow $2$-subgroup of $S_{n/16}$. If $G=A_n$ then we take $N\cap A_n$ for the
subgroup in question. Then $Y$ is a semidirect product of $N$ with $S_{16}$.
The latter permutes $P_1,\ldots,P_{16}$ in the natural way. One easily observes
that $|G:Y|$ is odd.

If $\Phi$ is a PIM of degree $|G|_2=|Y|_2$ then so is $\Phi|_Y$. By
\cite[Lemma 3.8]{Z11}, the generalised restriction $r^G_{Y/N}(\Phi)$ of $\Phi$
is a PIM of dimension $|Y/N|_2$. Such a PIM does not exist as we have just
seen.

Let $n>16$ be not a 2-power, and write $n=2^k+m$, where $m<2^k$. Let
$X=X_1\times X_2\le  S_n$, where $X_1\cong S_{2^k}$ and $X_2\cong S_m$.
Then the index $|S_n:X|$ is odd, so  $\Phi|_X$  is a PIM of degree $|X|_2$.
Therefore, $\Phi|_X$ is a direct product $\Phi_1\times \Phi_2$, where $\Phi_i$
is a PIM for $X_i$ for $i=1,2$. Obviously, $\dim\Phi_i=|X_i|_2$. This is a
contradiction, as $X_1$ has no PIM of degree $|X|_1$. For $G=A_n$ the result
follows from Lemma \ref{as7}.
\end{proof}

\section{Exceptional groups of Lie type}   \label{sec:exc}

\begin{thm}   \label{thm:exc}
 Let $G$ be a simple group of Lie type which is not classical. Then $G$ does
 not have a $\Syl_p$-regular character in non-defining characteristic, except
 for the group $G=\tw2F_4(2)'$ which has two reducible $\Syl_3$-regular
 characters and two irreducible Steinberg-like characters for $p=3$.
\end{thm}

\begin{proof}
As in the proof of \cite[Thm.~4.1]{MW} we compare the maximal order of a
Sylow $p$-subgroup of $G$, which is bounded above by the order of the
normaliser of a maximal torus, with the smallest irreducible character
degrees (given for example in \cite[Tab.~I]{TZ96}). This shows that except
for very small $q$ there cannot be any examples of $\Syl_p$-regular characters.
A closer inspection of the finitely many remaining cases shows that
$G=\tw2F_4(2)'$ has two reducible $\Syl_3$-regular characters and two
irreducible Steinberg-like character for $p=3$, but no further cases arise.
\end{proof}

\section{Groups of Lie type in their defining characteristic}   \label{sec:def char}

It was shown in \cite{PZ} that simple groups of Lie type of sufficiently large
rank don't have Steinberg-like characters with respect to the defining
characteristic apart from the (irreducible) Steinberg character. More
precisely, the Steinberg-like characters were classified except for groups
of types $B_n$ with $3\le n\le 5$ and $D_n$ with $n=4,5$.

Here we deal with the remaining cases. 

\begin{prop}   \label{prop:B3 defchar}
 Let $G=\Spin_{2n+1}(q)$, $n\in\{3,4,5\}$, with $q=p^f$ odd. Then $G$ has no
 reducible Steinberg-like character with respect to $p$.
\end{prop}

\begin{proof}
We freely use results and methods from \cite{PZ}. First assume that $n=3$.
According to \cite[Prop.~6.2]{PZ} it suffices to consider a group $H$ (coming
from an algebraic group with connected centre) such that $[H,H]=\Spin_7(q)$.
Let $\chi$ be a reducible Steinberg-like character of $H$. Then $\chi$ has a
linear constituent by \cite[Thm.~8.6]{PZ}. Multiplying by the inverse of that
character, we may assume that the trivial character occurs in $\chi$ (exactly
once). By \cite[Lemma~2.1]{PZ}, then all constituents of $\chi$ belong to the
principal $p$-block, so we may in fact replace $H$ by $H/Z(H)$, that is, we
may assume that $H$ is of adjoint type.  \par
Let $P\le H$ be a parabolic subgroup of $H$ and $U=O_p(P)$. Then by
Lemma~\ref{hc1}(c) the Harish-Chandra restriction $r_L^H(\chi)$ is a
Steinberg-like character of $L=P/U$. We will show that there is no possibility
for $\chi$ compatible with Harish-Chandra restriction to all Levi subgroups.
\par
Clearly, $r_L^H(\chi)$ also contains the
trivial character, so is again reducible. The reducible Steinberg-like
characters of all proper Levi subgroups of $H$ are known by \cite[Lemmas~7.4
and~7.8]{PZ}. In particular we must have $7|(q+1)$ and for $L$ of type $A_2$
we have $r_L^H(\chi)=1_L+\mu$ with $\mu\in\Irr(L)$ of degree~$q^3-1$.
Thus $\mu$ lies in the Lusztig series of a regular semisimple element
$s\in L^*$ (the dual group of $L$) with centraliser a maximal torus of order
$(q^2-1)(q-1)$. Thus $\chi$ has to contain a constituent $\psi_1$ lying in the
Lusztig series of $s$. It is easily seen that the centraliser of $s$ in $G^*$
is either a maximal torus, or of type $A_1(q).(q^2-1)$. Correspondingly,
$$\psi_1(1)\in\{(q^6-1)(q^2+1)(q+1),\ (q^6-1)(q^2+1),\ q(q^6-1)(q^2+1)\}.$$
But the first and the last are bigger than $q^9$, so $\psi_1(1)=(q^6-1)(q^2+1)$.
Now if $\chi$ contains any other constituent apart from $1_G$ in the principal
series, then its generalised restriction to $L$ is non-zero, contradicting
$r_L^H(\chi)=1_L+\mu$. \par
Next, the Harish-Chandra restriction to a Levi subgroup $L$ of type $B_2$ has
the form $r_L^H(\chi)=1_L+\nu_1+\nu_2+\nu_3$ with $\nu_1(1)=(q-1)^2(q^2+1)$ and
$\nu_2(1)=\nu_3(1)=(q-1)(q^2+1)$. In particular $\nu_1$ lies in the Lusztig
series of a regular semisimple element $t\in L^*$ (of order~7 dividing $q+1$)
with centraliser a maximal torus of order $(q^2-1)(q+1)$. The centraliser of
$t$ in $G^*$ then either is the same maximal torus, or of type $A_1(q).(q+1)^2$.
Correspondingly, $\chi$ has a constituent $\psi_2$ in the Lusztig series
of $t$ of degree
$$d_1:=(q^6-1)(q^2+1)(q-1)/(q+1),\,
  d_2:=q(q^6-1)(q^2+1)(q-1)/(q+1)\text{ or }(q^6-1)(q^2+1)(q-1).$$
The last one is larger than $q^9-1-\psi_1(1)$, so $\psi_2(1)\in\{d_1,d_2\}$. 
Furthermore, by \cite[Lemma~3.1]{PZ}, $\chi$ contains at least one
regular constituent. This is either $\psi_2$ of degree $d_2$,
or, if $\psi_2(1)=d_1$ 
then one can check from the known list of character degrees of $H$ (which can
be found at \cite{LuWeb}) that the only regular character $\psi_3$ of small
enough degree has degree $d_3:=(q^6-1)(q^2+1)(q-1)^2/(q+1)$. Observe that 
$d_2=d_1+d_3$. So the sum of remaining character degrees is 
$$d:=q^9-1-\psi_1(1)-d_2=q(q^3-1)(q^4-3q^3+3q^2-3q+1).$$
\par
Now note that $\chi$ cannot have further unipotent constituents since they
would lead to unipotent constituents of $r_L^H(\chi)$ (as $H$ has no cuspidal
unipotent characters). It turns out that all remaining candidates except for
one of degree $\lambda(1)=(q^3-1)(q^2+1)(q-1)$ have degree divisible by
$q^2-q+1$. Now $\lambda(1)\equiv2\pmod{q^2-q+1}$, while
$d\equiv-2\pmod{q^2-q+1}$. It follows that $\lambda$ would have to occur at
least $q^2-q$ times in $\chi$. As $(q^2-q)\lambda(1)>d$, this is not possible.
This contradiction concludes the proof for the case $n=3$.
\par
The cases of $\Spin_9(q)$ and $\Spin_{11}(q)$ now follow from the previous one
by application of the inductive argument in the proof of \cite[Thm.~10.1]{PZ}.
\end{proof}

\begin{prop}   \label{prop:D4 defchar}
 Let $G=\Spin_{2n}^+(q)$, $n\in\{4,5\}$, $q=p^f$. Then $G$ has no reducible
 Steinberg-like character with respect to $p$.
\end{prop}

\begin{proof}
First consider the case $n=4$. As in the previous proof, by
\cite[Prop.~6.2 and Lemma~2.1]{PZ} we may work with $H$ of adjoint type. Let
$\chi$ be a reducible Steinberg-like character of $H$. Then $\chi$ contains
$1_H$ by \cite[Thm.~8.6]{PZ} and hence so does its Harish-Chandra restriction
$r_L^H(\chi)$ to a Levi subgroup of type $A_3$. Then by \cite[Lemma~9.1]{PZ} we
have $q\equiv-1\pmod3$ and $r_L^H(\chi)=1_L+\mu_1+\mu_2+\mu_3$ with $\mu_1$
a cuspidal character labelled by a regular element $s\in L^*$ in a
torus of order $q^4-1$, of order dividing $(q^2+1)(q+1)$ in $L^*/Z(L^*)$.
But then $s$ is also regular in $H^*$, that is, $\chi$ has a constituent
$\psi$ of degree $(q^6-1)(q^4-1)(q^2-1)$. 
This holds for all three conjugacy classes of Levi subgroups of type $A_3$.
Comparison of degrees shows that this is not possible.
The case of $\Spin_{10}(q)$ again follows from the previous one by application
of the argument in the proof of \cite[Thm.~10.1]{PZ}.
\end{proof}

\section{Classical groups of large rank}   \label{sec:class large}

As an application of results obtained in Section~\ref{sec:alt} we show here
that classical
groups of large rank have no Steinberg like character for $p>2$, provided $p$
is not the defining characteristic of $G$. Throughout $p$ is an odd prime not
dividing $q$ and we set $e:=e_p(q)$, the order of $q$ modulo~$p$. We first
illustrate our method on the groups $\GL_n(q)$.

\begin{lem}   \label{ag9}
 Let $G=\GL_n(q)$, $p>2$, and let $S$ be a Sylow $p$-subgroup of $G$.
 \begin{enumerate}
  \item[\rm(a)] Write $n=me+m'$, where $0\leq m'<e$. Then there exist subgroups
   $U\le S \le N\le G$, where $U$ is an abelian normal $p$-subgroup of $N$ and
   $N/U\cong A_m$.
  \item[\rm(b)] If $m>\max\{6,p+1\}$ or $m'>0$ then $G$ has no Steinberg-like
   character.
 \end{enumerate}
\end{lem}

\begin{proof}
(a) See \cite{Weir}.
(b) If $m'>0$ then $G$ contains a subgroup $X$ such that $X\cong \GL_{m'}(q)$
and $C_G(X)$ contains a Sylow $p$-subgroup of $G$. As $X$ is a $p'$-group,
the result follows from Lemma~\ref{lem:gt7}. Let $m'=0$.
Suppose the contrary, and let $\chi$ be a Steinberg-like character of $G$.
By Lemma~\ref{hc1}, $A_m$ must have a Steinberg-like character. However,
this is false by Theorem~\ref{thm:An}.
\end{proof}

For other classical groups the argument is similar, but involves more technical
details. Let $d=e_p(-q)$ be the order of $-q$ modulo $p$, equivalently,
$d=2e$ if $e$ is odd, $d=e/2$ if $e\equiv 2\pmod 4$, and $d=e$ if $4|e$.
So $d=1$ if and only if $e=2$, equivalently, $p|(q+1)$. Note that
$e=2e_p(q^2)$ if $e$ is even.

\begin{lem}   \label{nb8}
  {\rm \cite{Weir}} Let $G=\GU_n(q)$ and $p>2$. Then the Sylow $p$-subgroups
   of $G$ are isomorphic to those of $H$, where $H\cong\GL_{\flhalfn}(q)$ if
   $e$ is odd, $H\cong\GL_{\flhalfn}(q^2)$ if $4|e$ and $H\cong \GL_n(q^2)$ if
   $e\equiv 2\pmod 4$.
\end{lem}



\begin{lem}   \label{nu7}
 Let $G=\GU_n(q)$, $p>2$, and let $S$ be a Sylow $p$-subgroup of $G$.
 Suppose that $e\equiv 2\pmod 4$, equivalently, $d$ is odd.
 \begin{enumerate}
  \item[\rm(a)] Write $n=md+m'$, where $0\leq m'<d$. Then there exist subgroups
   $U\le S \le N\le G$, where $U$ is an abelian normal $p$-subgroup of $N$ and
   $N/U\cong A_m$.
  \item[\rm(b)] If $m>\max\{6,p+1\}$ or $m'>0$ then $G$ has no Steinberg-like
   character.
 \end{enumerate}
\end{lem}

\begin{proof}
(a) Suppose first that $e=2$. Let $V$ be the natural $\FF_{q^2}H$-module.
Then $V$ is a direct sum $\oplus_{i=1}^nV_i$, where $V_i$'s are non-degenerate
subspaces of dimension~1. Let $X$ be the stabiliser of this decomposition,
that is, $X=\{x\in G\mid xV_i=V_j$ for some $j=j(x)\in\{1,\ldots ,n\}\}$.
Then $X\cong X_1\cdot S_n$ (a semidirect product), where
$X_1\cong (\GU_1(q)\times \cdots \times \GU_1(q))$ ($n$ factors). Let $U$
be the Sylow $p$-subgroup of $X_1$. Then $U$ is normal in $X$ and abelian.
It is well known that $X$ contains a Sylow $p$-subgroup of $G$. Therefore, $N=UA_n$ satisfies the statement.
\par
Let $e>2$. As $d$ is odd, there is an embedding
$\GU_m(q^d)\rightarrow\GU_{md}(q)$ (see \cite[Hilfssatz~1]{Hu}). Note that
$e_p(q^d)=2$ and $|\GU_m(q^d)|_p=|\GU_{md}(q)|_p$.
As $\GU_{md}(q)$ is isomorphic to a subgroup of $G$, the result follows.

(b) is similar to the proof of Lemma \ref{ag9}(b).
\end{proof}

\begin{lem}   \label{c5a}
 Let $p>2$, $2n=me$, where $e=e_p(q)$ is even, and $X=\GU_m(q^{e/2})$.
 \begin{enumerate}[\rm(a)]
  \item If $m$ is even (resp.~odd) then $X$ is isomorphic to a subgroup of
   $\GO_{2n}^+(q)$ (resp. $\GO_{2n}^-(q)$).
  \item $X$ is isomorphic to a subgroup of $\Sp_{2n}(q)$, of $\GO_{2n+1}(q)$,
   of $\GO_{2n+e}^+(q)$ as well as of $\GO_{2n+e}^-(q)$.
 \end{enumerate}
 In addition, $X$ contains a Sylow $p$-subgroup of the respective group.
\end{lem}

\begin{proof}
(a) follows from \cite[Lemma 6.6]{EZ} as well as (b) for $\Sp_{2n}(q)$.
The second case in (b) follows from (a) as the groups $\GO_{2n+1}(q)$,
$\GO_{2n+e}^+(q)$ and $\GO_{2n+e}^-(q)$ contain subgroups isomorphic to
$\GO_{2n}^+(q)$ and $\GO_{2n}^-(q)$.

The additional statement can be read off from the orders of the groups in
question. (The cases with $\Sp_{2n}(q)$, $\GO^-_{2n}(q)$ and $\GO_{2n+1}(q)$
are considered in \cite[Lemmas~3.14 and~3.16]{HSZ}, that of $\GO^+_{2n}(q)$
is similar.)
\end{proof}

\begin{lem}   \label{ag8}
 Let $p>2$. Let $H$ be one of the following groups:
 \begin{enumerate}
  \item[\rm(1)] $H=\GU_n(q)$ with $n=md+m'$, where $m'<d$;
  \item[\rm(2)] $H\in\{\Sp_{2n}(q),\GO_{2n+1}(q),\GO_{2n}^+(q),
   \GO_{2(n+1)}^-(q)\}$ with $n=me+m'$, where $e$ is odd and $m'<e$;
  \item[\rm(3)] $H\in\{\Sp_{2n}(q),\GO_{2n+1}(q)\}$ with $2n=me+m'$ where
   $e$ is even and $m'<e$;
  \item[\rm(4)] $H=\GO_{2n}^\pm(q)$ with $2n=me+m'$ where $e$ is even, $m'<e$,
   and either $m'>0$, or $m'=0$ and then either $H=\GO_{2n}^+(q)$, $m$ is
   even, or $H=\GO_{2n}^-(q)$, $m$ is odd;
  \item[\rm(5)] let $e$ be even, $2n=(m+1)e$ and $H=\GO_{2n}^+(q)$,
   $m+1$ is odd, or $H=\GO_{2n}^-(q)$ and $m+1$ is even.
 \end{enumerate}
 Let $S$ denote a Sylow $p$-subgroup of $H$. Then there exist subgroups
 $U\le S \le P\le H$, where $U$ is an abelian normal $p$-subgroup of $P$ and
 $P/U\cong A_m$.
\end{lem}

\begin{proof}
(1) The case $e\equiv 2\pmod 4$ is handled in Lemma~\ref{nu7}.
In the remaining cases the result follows from Lemmas~\ref{ag9} and~\ref{nb8},
as $\GL_{[n/2]}(q^2)$ is isomorphic to a subgroup of $G$.

(2) By Lemma \ref{nb8}, $|H|_p=|\GL_n(q)|_p$. So the result follows from
Lemma~\ref{ag9}.

(3) This follows from Lemmas~\ref{c5a} and~\ref{ag9}. (Note that
$H=\GO_{2n+1}(q)$ contains subgroups isomorphic to $\GO_{2n+1}^+(q)$ and
$\GO_{2n}^-(q)$ and one of them contains a Sylow $p$-subgroup of $H$.)

(4) Similar to (3). Note that if $m'>0$ then $H$ contains subgroups isomorphic
to $\GO_{me}^+(q)$ and $\GO_{me}^-(q)$, and one of them contains a Sylow
$p$-subgroup of $H$.

(5) In this case a subgroup $X$ of $H$ isomorphic to $\GO_{2n-e}^+(q)$ and
$\GO_{2n-e}^-(q)$, respectively, contains a Sylow $p$-subgroup of $H$. (One can
easily check that $|H:X|_p=1$.) So the result follows from (4).
\end{proof}

Our result for alternating groups (Corollary~\ref{p11}) implies the following:

\begin{prop}   \label{t5r6}
 Let $p>2$ and $e=e_p(q)$. Let $m=\max\{7, p+2\}$. Let $G$ be one of the
 following groups:
 \begin{itemize}
  \item[(1)] $\PSL_n(q)$ and $n\geq em$;
  \item[(2)] $\PSU_n(q)$ and $n\geq dm$, where $d=e_p(-q)$;
  \item[(3)] $\Omega_{2n+1}(q)$, $q$ odd, or $\PSp_{2n}(q)$, $n>1$,
   and $n\geq em$ if $e$ is odd, otherwise $2n\ge em$;
  \item[(4)] $\POm_{2n}^+ (q)$, and $n\geq em$ if $e$ is odd, otherwise
   $2n\ge em$;
  \item[(5)] $\POm_{2n}^-(q)$, and $n-1\geq em$ if $e$ is odd, otherwise
   $2n\geq em$.
 \end{itemize}
 Then $G$ has no Steinberg-like character. This remains true for any group $H$
 such that $G$ is normal in $H/Z(H)$ and $(H/Z(H))/G$ is abelian.
\end{prop}

\begin{proof}
Suppose first that $H$ is as in Lemma \ref{ag8}. Let $S\in \Syl_p(H)$.
Then there are subgroups $U\le S\le P\le H$, where $U$ is normal in $P$ and
$N/U\cong A_m$ with $m\geq  m=\max\{7, p+2\}$. So $A_m$ is perfect. Let $H_1$
be the
derived subgroup of $H$. Set $P_1=P\cap H_1$, $S_1=S\cap H_1$, $U_1=U\cap H_1$.
Then $S_1\in\Syl_p(H_1)$,  $U_1\le S_1 \le P_1\le H_1$ and $P_1/U_1\cong A_m$,
as $A_m$ is perfect. A similar statement is true for the quotient of $H_1$
by a central subgroup. Then the result follows from Theorem~\ref{thm:An}
using Lemma \ref{hc1}.
\end{proof}


\section{Minimal characters and Sylow $p$-subgroups, $p>2$}   \label{sec:class p>2}

In this section we show that if $p>2$ and $G$ is a simple classical group not
satisfying the assumptions in Proposition~\ref{t5r6}, $p$ is not the defining
characteristic of $G$ and a Sylow $p$-subgroup $S$ of $G$ is not cyclic,
then $G$ has no $\Syl_p$-regular character and hence no Steinberg-like
character.
Observe that $S$ is cyclic if and only if $m=1$, and abelian if and only if
$m<p$, where $m$ is as in Lemma \ref{ag8}. The case where $S$ is cyclic has
been dealt with in Section~\ref{sec:cycSyl}.

For a group $G$, let $\mu_0(G)=1<\mu_1(G)<\mu_2(G)<\cdots$ denote the sequence
of integers such that for $i>0$, $G$ has an \ir character of degree $\mu_i(G)$
and no \ir character $\rho$ with $\mu_{i-1}(G)<\rho(1)<\mu_i(G)$. For universal
covering groups of finite classical groups the values
$\mu_1(G),\mu_2(G),\mu_3(G)$ were determined in \cite{TZ96}. In our analysis
below these three values play a significant role, but mainly for classical
centerless groups $G$ such as $\PGL_n(q)$, $\PGU_n(q)$, $\PSp_{2n}(q)$,
$\POm_{2n}^\pm(q)$ and $\Omega_{2n+1}(q)$. For these groups, mainly for
$2e\leq n\leq pe$ with $p>2$, we observe that $|G|_p<\mu_3(G)$ and sometimes
$|G|_p<\mu_1(G)$. In the latter case it is immediate to conclude that $G$ has
no $\Syl_p$-regular character, in the other cases we observe that there exists
an element $g\in G$ of order $p$ such that $\rho(g)>0$ for each \ir character
$\rho$ of degree at most $\mu_2(G)$. For $n>pe$ we use a different method.

Recall that $e=e_p(q)$ denotes the minimal integer $i>0$ such that $q^i-1$
is divisible by $p$.

\subsection{The groups $\GL_{n}(q), n\geq 2e$ }
Set $d_n=(q^n-1)/(q-1)$. Let $G=\SL_n(q)$. The minimal degrees of projective
\ir representations of $\PSL_n(q)$ are given in \cite[Table~IV]{TZ96}.
Table~\ref{tab:min SLn} is obtained from this by omitting the representations
that are not realisable as ordinary representations of $\SL_n(q)$.

\begin{table}[ht]
 $$\begin{array}{|l|ccc|}
 \hline
  n,q& \mu_1(G)& \mu_2(G)& \mu_3(G)\\
 \hline
 n=3,\ q>2& d_3-1& d_3& (q^2-1)(q-1)/(3,q-1)\\
 n=4,\ q>3& d_4-1& d_4& (q^3-1)(q-1)/(2,q-1)\\
 n=4,\ q=3& 26& 39& 52\\
 n>4,\ q>2, (n,q)\neq (6,3)& d_n-1& d_n& d_n(q^{n-1}-q^2)/(q^2-1)\\
 n>4,\ q=2,\ n\neq 6& d_n-1& d_n(2^{n-1}-4)/3& d_nd_{n-1}/3\\
 n=6,\ q=2& 62& 217& 588\\
 n=6,\ q=3& 363& 364& 6318\\
 \hline
 \end{array}$$
\caption{Minimal degrees of irreducible characters of $\SL_n(q)$}   \label{tab:min SLn}
\end{table}

\begin{lem}   \label{dd77}
 Let $p>2,e>1$, and $G= \GL_{en}(q)$. Suppose that $1<n\leq p$, and
 if $q=2$ suppose that either $n<p$ or $p<2^e-1$. Then $|G|_p<\mu_1(G)$
 and $G$ has no $\Syl_p$-regular character.
\end{lem}

\begin{proof}
If $n<p$ then $|G|_p\leq (q^e-1)^n/(q-1)^n$ (as $p$ is coprime to $q-1$),
and $\mu_1(G)=(q^{en}-q)/(q-1)$. So the statement is obvious in this case.

Let $n=p$. If $q>2$ then $|G|_p\leq\frac{p(q^e-1)^p}{(q-1)^p}$, while
$\mu_1(G)=\frac{q^{ep}-q}{q-1}$. (As $p>2$, the exceptions in
Table~\ref{tab:min SLn} can be ignored, except for $e=3$ and $(n,q)=(6,2)$
or $(6,3)$; these two cases are trivial.) We have
$p(q^e-1)^p<(q-1)^{p-1}(q^{ep}-q)$ as $p<(q-1)^{p-1}$ for $q>2$ and
$(q^e-1)^p<q^{ep}-q$.

If $q=2$ and $p<2^e-1$ then $p\leq (2^e-1)/3$ and $|S|\le p(2^e-1)^p/3^p$ is
less than $2^{ep}-2=\mu_1(G)$ as $p<3^p$.
\end{proof}

\begin{rem}   \label{666}
Lemma~\ref{dd77} does not extend to the case $q=2$ with $n=p=2^e-1$ as then
$$|G|_p=p(2^e-1)^p=p^{p+1}>(p+1)^p-2=2^{ep}-2=\mu_1(G).$$
So the case $e>1$ leaves us with $q=2$, which we deal with next.
\end{rem}

\begin{lem}   \label{u44}
 Let $e>1$ and $G=\GL_{ep}(2)$.
 Then $|G|_p<\mu_2(G)$ and $G$ has no $\Syl_p$-regular character.
\end{lem}

\begin{proof}
We have $|G|_p\leq p(2^e-1)^p $. By Table~\ref{tab:min SLn} we have
$\mu_1(G)=2^{ep}-2$ and $\mu_2(G)=(2^{ep}-1)(2^{ep-1}-4)/3>|G|_p$, or
$ep=3,4,6$. As $p$ is odd
and $e>1$, we have $ep\neq 3,4$, so in the exceptional cases $e=2$, $p=3$ where
$|G|_p=81<\mu_2(G)=217$.

Let $\pi$ be the permutation character of $G$ associated with the action of
$G$ on the non-zero vectors of the natural $\FF_2G$-module. Then $\pi=\tau+1_G$,
where $\tau$ is a character of $G$ of degree $\tau(1)=2^{ep}-2$. There is a
unique irreducible character of $G$ of degree $2^{ep}-2$ \cite[Table~IV]{TZ96},
and hence it coincides with $\tau$.
Let $\chi$ be a $\Syl_p$-regular character of $G$. As $\chi(1)=|G|_p$, it
follows that the irreducible constituents of $\chi$ are either $1_G$ or
$\tau$. As $\chi(g)=0$ for every $p$-element $g\in G$, we get a contradiction
as soon as we show that $\tau(g)>0$ for some $p$-element $g\in G$. This is
equivalent to showing that $\pi(g)>1$. This can be easily verified.
\end{proof}

\begin{lem}   \label{md5}
 Suppose that $p>2$, $p|(q-1)$ and let $\SL_n(q)\le G\le \GL_n(q)$ for $2<n<p$.
 Then $G$ has no $\Syl_p$-regular character.
\end{lem}

\begin{proof}
Let $G_1=G/Z(G)$. Then $|G_1|_p\leq |q-1|^{n-1}_p$ and
$\mu_1(G_1)=(q^n-q)/(q-1)>q^{n-1}$ as above. So $|G_1|_p<\mu_1(G_1)$, and
$G_1$ has no $\Syl_p$-regular character. Then neither has $G$ by
Lemma~\ref{hc1}.
\end{proof}

\subsection{The groups $\GU_{n}(q)$, $n>2$}

In this section we consider the case where $p>2$ and a Sylow $p$-subgroup of
$\GU_n(q)$ is abelian or abelian-by-cyclic. This implies $n<dp^2$, where $d$
is the order of $-q$ modulo $p$, equivalently, $d=2e$ if $e$ is odd, $d=e/2$
if $e\equiv 2\pmod 4$, and $d=e$ if $4|e$. So $d=1$ if and only if $e=2$,
equivalently, $p|(q+1)$. Note that $e=2e_p(q^2)$ if $e$ is even.

\begin{lem}   \label{na9}
 Let $p$ be odd, let S be a Sylow $p$-subgroup of $G=\GU_{n}(q)$.
 Then $S$ is abelian and not cyclic if and only if $2d\leq n<dp$.
\end{lem}

\begin{proof}
If $X=\GL_n(q)$ then Sylow $p$-subgroups of $X$ are abelian if and only if
$n<ep$. Let $S$ be a Sylow $p$-subgroup of $G$. We use Lemma~\ref{nb8}.
If $e\equiv 2\pmod 4$ then $|G|_p=|\GL_n(q^2)|_p$, so $S$ is abelian if and only
if $n<e_p(q^2)p=ep/2=dp$. If $e$ is odd then $|G|_p=|\GL_{\flhalfn}(q)|_p$, so
$S$ is abelian if and only if $\flhalfn<ep$, equivalently, $n<2ep=dp$.
If $4|e$ then $|G|_p=|\GL_{\flhalfn}(q^2)|_p$, so $S$ is abelian if and only if
$\flhalfn<e_p(q^2)p$, equivalently, $n<ep=dp$. Similarly, $S$ is cyclic if and
only if $n< 2d$. So the lemma follows.
\end{proof}

\begin{lem}   \label{lem:SU_dp}
 Let $p>2$, $d>1$ and $\SU_{dp}(q)\le G\le \GU_{dp}(q)$.
 Then $|G|_p<\mu_1(G)$ and $G$ has no $\Syl_p$-regular character. This remains
 true if $\SU_{n}(q)\le G\le \GU_{n}(q)$ with $2d\leq n<dp$.
\end{lem}

\begin{proof}
Note that $d>1$ means that $p$ does not divide $q+1$, so $|G|_p=|\SU_d(q)|_p$.
So it suffices to prove the lemma for $G=\SU_{dp}(q)$. First assume that $d$
is odd, so $e\equiv 2\pmod 4$ and $d=e/2$.
Then we have $|G|_p\leq p(q^{e/2}+1)^p/(q+1)^p=p(q^{d}+1)^p/(q+1)^p$
and $\mu_1(G)=(q^{dp}-q)/(q+1)$. In this case $|G|_p<\mu_1(G)$. Similarly,
if $e$ is odd then $d=2e$,  $|G|_p\leq p(q^{d}-1)^p/(q+1)^p$ and
$\mu_1(G)=(q^{dp}-1)/(q+1)$, so again $|G|_p<\mu_1(G)$. \par
Finally assume that $4|e$. Then $d=e$, $|G|_p\leq p(q^e-1)^p/(q^2-1)^p$
and $\mu_1(G)=(q^{ep}-1)/(q+1)$. So $|G|_p<\mu_1(G)$ again. This implies that
$G$ has no $\Syl_p$-regular character.

The proof of the additional statement is similar.
\end{proof}

Thus, we are left with primes such that $p|(q+1)$. We first consider the case
where $n<p$.

\begin{lem}   \label{u32}
 Let $p|(q+1)$, $2<n<p$, and $\SU_{n}(q)\le G\le \GU_{n}(q)$. Then $G$
 has no $\Syl_p$-regular character.
\end{lem}

\begin{proof}
Let $G_1=\{g\in G\mid \det g$ is an element of $\GU_1(q)$ of $p'$-order$\}$.
Then $|G:G_1|=|Z_p|$, where $Z_p$ is the Sylow $p$-subgroup of $Z(G)$.
As $Z_p\cap G_1=1$, it follows that $G=G_1\times Z_p$.
By Lemma~\ref{da3}, the result for $G$ follows if we show that $G_1$ has no
$\Syl_p$-regular character. In turn, this follows from the same result for
$G'=\SU_n(q)$ as $|G_1:G'|$ is coprime to $p$. So we deal with $G'$.

Suppose the contrary, and let $\chi$ be a $\Syl_p$-regular character of $G'$.

First, let $n=3$.  Then $\chi(1)=|G'|_p=(q+1)_p^2$ for $p>3$. By
\cite[Table~V]{TZ96}, $\mu_1(G')=q^2-q$. Let $q>4$. Then
$\chi(1)=|G'|_p=(q+1)_p^2<2(q^2-q)=2\mu_1(G')$. So $\chi$ has a single
non-trivial \ir constituent $\rho$, and $\rho(1)\leq \chi(1)$. Again by
\cite[Table~V]{TZ96}, $q^2-q\leq \rho(1)\leq q^2-q+1$. As $\chi$ is
$\Syl_p$-regular, $(\chi,1_{G'})\leq 1$ (Lemma~\ref{lem:a51}).
Therefore, $\chi(1)\leq \rho(1)+1\leq q^2-q+2$, which is false. The case with
$q=4$ can be read off from the character table of $G'$.

Let $n>3$ and let $V$ be the natural $\FF_{q^2}G'$-module. Let $b_1,\ldots,b_n$
be an orthogonal basis in $V$ and let $W=\lan b_1,b_2,b_3\ran$. Then $W$ is a
non-degenerate subspace of $V$ of dimension 3.
Set $X=\{h\in G'\mid hW=W$ and $hb_i\in\lan b_i\ran$ for $i=4,\ldots,n\}$ and
$U:=O_p(X)$. Then $U\subseteq Z(X)$ and every element of $U$ acts scalarly on
$W$. Let $X'$ be the derived subgroup of $X'$ and $P=X'U$. Then
$X'\cong \SU_3(q)$ and $P/U\cong \SU_3(q)$ (as $p>3$).
By the above, $\SU_3(q)$ has no $\Syl_p$-regular character. As $P$ contains
a Sylow $p$-subgroup of $G'$, the result follows from Lemma~\ref{hc1}.
\end{proof}

\begin{lem}   \label{upq2}
 Let $\SU_3(8)\le G\le \GU_3(8)$. Then $G$ has no $\Syl_3$-regular character.
\end{lem}

\begin{proof}
By Lemma~\ref{hc1}, it suffices to prove that $\PSU_3(8)$ and $\PGU_3(8)$
have no $\Syl_3$-regular character. Suppose the contrary, and let $\chi$ be a
$\Syl_3$-regular character of any of these groups. As $|\PGU_3(8)|_3=243$, we
have $\chi(1)\leq 243$. Let $\rho$ be an \ir constituent of $\chi$. Then
$\rho(1)\leq 243$. By \cite{Atl}, $\rho(1)\in\{1,56,57,133\}$, and the
characters of degree 1,56,133 are positive at the class $9A$, whereas those
of degree $57$ vanish at this class. It follows that $\rho(1)=57$, but then
$\rho$ is positive at the class $3C$. This is a contradiction.
\end{proof}

\begin{lem}   \label{upq1}
 Let $H=\SU_p(q)$, $p>2$, $p|(q+1)$, or $H=\SL_p(q)$, $p>2$, $p|(q-1$), and
 $h=\diag(1,\ep,\ep^2,\ldots,\ep^{p-1})\in H$, where $\ep\in\FF_{q^2}^\times$
 is a primitive $p$-th root of unity. Let $\chi$ be an irreducible character
 of $H$ whose kernel has order prime to $p$. Then $\chi(h)=0$.
\end{lem}

\begin{proof}
The element $h$ is written in an orthogonal basis of the underlying vector
space in the unitary case. Then  $h\in E$, where $E\le H$ is an extraspecial
group of order $p^3$ such that $Z(E)=Z(H)$. The restriction of $\chi$ to $E$
is a direct sum of
irreducible representations of $E$ non-trivial on $Z(E)$. It is well known and
can be easily checked that the character of every such representation vanishes
at $h$. So the claim follows.
\end{proof}

Let $H=\GU_n(q)$ or $\GL_n(q)$ with $n>2$. Weil representations of these groups
were studied by Howe \cite{Ho} and other authors, and have many applications,
mainly due to the fact that their \ir constituents (which we call \ir Weil
representations) essentially exhaust the \ir representations of degree
$\mu_1(H)$ and $\mu_2(H)$. More details are given below for $n=p$, $p$ odd.
Let $M$ be the underlying space of the Weil representation of $H$. Then
$M=\oplus_{\zeta\in\Irr (Z(H))} M_\zeta$, where
$M_\zeta=\{m\in M\mid zm=\zeta(z)m $ for $z\in Z(H)\}$.
In general, $H$ is irreducible on $M_\zeta$, except  for the case where
$H=\GL_p(q)$ and  $\zeta=1_{Z(H)}$. In this case $M_\zeta$ is a sum of a
one-dimensional and an irreducible $H$-invariant subspace.

So the \ir Weil representations $\rho$ of $H$ of dimension greater than~1 are
parameterised by their restriction to $Z(H)$, and each of them remains \ir
under restriction to $H'=\SU_n(q)$ or $\SL_n(q)$. By \cite{TZ96}, every
irreducible representation of $H'$ of degree $\mu_1(H)$ and $\mu_2(H)$ is an
\ir Weil representation. Moreover, every irreducible representation of $H$ of
degree $\mu_1(H)$ and $\mu_2(H)$ is obtained from an \ir Weil representation
by tensoring with a one-dimensional representation.

\begin{lem}  \label{et1}
 Let $p>2$, and $H=\GU_p(q)$, $p|(q+1)$, $(p,q)\neq (3,2)$, or $\GL_p(q)$,
 $p|(q-1)$. Let $\zeta\in\Irr(Z(H))$.
 Let $\rho=\rho_\zeta$ be the character of an \ir constituent of the Weil
 representation $\om$ of $H$ labeled by $\zeta$ (where $\rho(1)>1$). Let $h$
 be as in Lemma~{\rm\ref{upq1}}. Then $\rho(h)\in\{0,p,p-1\}$, except for
 the case with $G=\GL_p(q)$ and $\zeta=1_{Z(H)}$, where $\rho(h)=p-2$.
 In addition, $\rho(h)\neq0$ if and only if $\rho(z)=1$ for an element
 $z\in Z(H)$ of order $p$.
\end{lem}

\begin{proof}
We only consider the case $H=\GU_p(q)$, as the case $H=\GL_p(q)$ is similar.

Let $Z=Z(H)$, $\zeta\in\Irr(Z)$ and $\rho=\rho_\zeta$ be the \ir constituent
of $\om$ labeled by $\zeta$. This means that $\rho(z)=\rho(1) \zeta(z)$.

Let $X=\lan Z,h\ran$. Let $\ep_i$ be the character of $\lan h \ran$ such that
$\ep_i(h)=\nu^i$, where $\nu$ is a fixed $p$th root of unity, $i=1,\ldots,p$.
Then the multiplicity of the eigenvalue $\nu^i$ of $\rho(h)$ equals
$(\om|_{X}, \zeta\cdot \ep_i)$.

Recall that $\om(x)=-(-q)^d$, where $d$ is the multiplicity of the eigenvalue~1
of $x$ as a matrix in $\GU_p(q)$. Therefore, $\om(1)=q^n$, and if $x=zh^k$
and $z^p\neq 1$ then $\om(x)=-1$; if $z^p=1,h\neq 1$ then $\om(x)=q$ (also
for $z=1$).

We compute $|X|\cdot(\om|_{X},\zeta\cdot\ep_i)
=\sum_{x\in X}\om(x)\zeta(z)\ep_i(h)$, where $x=zh$. Note that $\om(x)$ is an
integer, so $\om$ is self-dual.
Let $Z_p$ be the subgroup of order $p$ in $Z$, and $X_p=\lan Z_p,h\ran$. Then
$$\sum_{x\in X}\om(x)\zeta(z)\ep_i(h)=\sum_{x\in X_p} \om(x)\zeta(z)\ep_i(h)
   ~~+~~\sum_{x\notin X_p} \om(x)\zeta(z)\ep_i(h).$$
We first show that the second sum equals 0 if $i<p$. Note that $x=zh\notin X_p$
is equivalent to $z\notin Z_p$. Therefore, $d=d(x)=0$ for $x\notin X_p$, and
$\om(x)=-1$. For $z$ fixed we have a partial sum $\zeta(z)\sum_h\ep_i(h)$,
and $\sum_h\ep_i(h)=0$, as claimed.

If $i=p$ and $\zeta\neq 1_Z$ then
$$\sum_{x\notin X_p} \om(x)\zeta(z)\ep_i(h)=-p\sum_{z\notin Z_p} \zeta(z)
  =-p(\sum_{z\in Z}\zeta(z)-\sum_{z\in Z_p}\zeta(z))=p^2$$
as $\zeta(z)=1$ for $z\in Z_p$.  If $i=p$ and $\zeta=1_Z$ then
$$\sum_{x\notin X_p} \om(x)\zeta(z)\ep_i(h)=-(|X|-|X_p|)=-p(q+1)+p^2.$$

Next, we compute $\sum_{x\in X_p} \om(x)\zeta(z)\ep_i(h)$. Observe that
$\zeta(z)=1$ for $z\in Z_p$ so this sum simplifies to
$\sum_{zh\in X_p} \om(zh)\ep_i(h)$. Note that  $d(zh)=1$ if $h\neq 1$ and any
$z\in Z_p$. So if $h\neq 1$ then $\om(zh)=q$.
If $h=1$ then $d(zh)=d(z)=0$ for $z\neq 1$ so $\om(z)=-1$. And $\om(1)=q^p$.

Therefore, we have
$$\begin{aligned}
 \sum_{zh\in X_p} \om(zh)\ep_i(h)&=\sum_{z\in Z_p,h\neq 1}\om(zh)\ep_i(h)
  +\sum_{z\in Z_p,z\neq 1} \om(z)+q^p\\
  &=\sum_{z\in Z_p,h\neq 1} q\cdot\ep_i(h)+\sum_{z\in Z_p,z\neq 1} (-1)+q^p
  =pq\sum_{ h\neq 1} \ep_i(h)-(p-1)+q^p.
\end{aligned}$$

(i) Let $i\neq p$. Then $\sum _{h\neq 1}\ep_i(h)=-1$, and the last sum equals
$-pq-p+1+q^p=q^p+1-p(q+1)$.

(ii) Let $i=p$. Then $\sum _{h\neq 1}\ep_i(h)=p-1$, and the last sum equals
$pq(p-1)-(p-1)+q^p=q^p+1+p^2q-pq-p$.

Therefore, $|X|\cdot (\om|_{X}, \zeta\cdot \ep_i)=q^p+1-p(q+1)$ if $i\neq p$,
$q^p+1+(p-1)p(q+1)$ if $i=p$, $\zeta\neq 1_Z$, and $q^p+1+(p-2)p(q+1)$ if
$i=p$, $\zeta=1_Z$.

In particular, the multiplicities of eigenvalue $\nu^i$ for $i\neq p$ of $h$
on the module $M_\zeta$ for fixed $\zeta$ are the same. As
$\sum_{i\neq p}\nu^i=-1$, the trace of $h$ on $M_\zeta$ for $\zeta\neq 1_Z$
with $\zeta(Z_p)=1$ equals
$$(1/|X|)(q^p+1+(p-1)p(q+1)-(q^p+1-p(q+1))=p$$
as $|X|=p(q+1)$. Similarly, if $\zeta=1_Z$ then the trace in question equals
$p-1$. In other words, if $\om_\zeta$ is the character of $M_\zeta$ and
$\zeta(Z_p)=1$ then $\om_\zeta(h)=p$ for $\zeta\neq 1_Z$ and $p-1$ otherwise.
\end{proof}

\begin{lem}  \label{et2}
 Let  $H=\GU_p(q)$,  $p|(q+1)$, $p>2,(p,q)\neq (3,2)$ or $\GL_p(q)$, $p|(q-1)$.
 Then  H  has no $\Syl_p$-regular character. The same is true for $H'=\SU_p(q)$
 and $\SL_p(q)$ and for all groups $X$ with $H'\le X\le H$.
\end{lem}

\begin{proof}
Set $G=H/O_p(Z(H))$. By Lemma~\ref{hc1}, it suffices to prove the lemma for $G$
in place of $H$.

Suppose the contrary, and let  $\chi$ be a $\Syl_p$-regular character of $G$,
and let $\lam$ be an \ir constituent of $\chi$. We first observe that
$\lam(1)< \mu_3(G)$, and hence by \cite{TZ96},
$\lam(1)\in\{1,\mu_1(G),\mu_2(G)\}$.

Indeed, note that $|G|_p=p |q+1|_p^{p-1}$ in the unitary case, respectively
$|G|_p= p|q-1|_p^{p-1}$ in the linear case. By \cite[Table~IV]{TZ96},
$\mu_3(G)\geq \mu_3(H')\geq \frac{(q^p+1)(q^{p-1}-q^2)}{(q+1)(q^2-1)}$, resp.,
$\frac{(q^p-1)(q^{p-1}-q^2)}{(q-1)(q^2-1)}$ if $p>3$. This value is greater
than $|G|_p$. Let $p=3$. Then $\mu_3(G)\geq \mu_3(H')\geq (q^2-q+1)(q-1)/3$,
resp., $(q^2-1)(q-1)/3$ for $q>4$. Again, $|G|_3<\mu_3(G)$, unless
$G=\PGU_3(8)$ or $\PGL_3(4)$. The former case is settled in Lemma~\ref{upq2}.

Let $G=\PGL_3(4)$. Then $|G|_3=27$. In this case $\mu_1(G)=20$, $\mu_2(G)=35$
and $\mu_3(G)=45$. So $\lam(1)\leq |G|_p$ implies $\lam(1)\leq 20$. The
character of degree~20 is positive at class $3A$, a contradiction.

So $|G|_p\leq \mu_2(G)$. As mentioned prior Lemma~\ref{et1}, $\lam$ is either
one-dimensional or can be seen as a character of $H$ obtained from an \ir Weil
character by tensoring with a linear character of $H$. Let $h\in H$ as in
Lemma~\ref{et1}. Then $h\in H'$, so tensoring can be ignored, and we can
assume that $\lam$ is an \ir Weil character of $H$. Then, by Lemma~\ref{et1},
$\lam(h)\in\{0,p,p-1\}$ in the unitary case and  $\lam(h)\in\{0,p,p-2\}$
in the linear case. If $\lam(1)=1$ then $\lam(h)=1$. So $\lam(h)\geq 0$.
As $\chi$ is $p$-vanishing and $|h|=p$, we have $\chi(h)=0$. So $\lam(h)=0$
for every \ir constituent of $\chi$. This is false as $\lam$ is trivial on
$O_p(Z(H))$ by the definition of $G$, and hence $\lam(h)\neq 0$ by
Lemma~\ref{et1}. This is a contradiction. As irreducible Weil representations
of $H$ remain irreducible upon restriction to $H'$, this argument works for
intermediate groups $X$ too.
\end{proof}

\begin{lem}   \label{8ee}
 Let $p>2$ and let $G$ be a group such that $\SL_n(q)\le G\le \GL_n(q)$
 with $2e<n\leq ep$, or $\SU_n(q)\le G\le\GU_n(q)$ with $2d<n\leq dp$ and
 $(n,q)\neq(3,2)$. Then $G$ and $G/O_p(G)$ have no $\Syl_p$-regular character.
\end{lem}

\begin{proof}
For the unitary case with $d>1$ the result for $G$ is stated in
Lemma~\ref{lem:SU_dp}. The case with $d=1$ and $n<p$ is dealt with in
Lemma~\ref{u32}, and the remaining case $d=1$ and $n=p$ is examined in
Lemma~\ref{et2}.

Let $H=\GL_n(q)$. The result for $e>1,q>2$ follows from Lemma~\ref{dd77},
and that for $e>1,q=2$ is proved in Lemma~\ref{u44}. The result for $e=1,n=p$
is stated in Lemma~\ref{et2}. The case with $e=1,n<p$ is examined in
Lemma~\ref{md5}.

The statement on $G/O_p(G)$ follows from Lemma \ref{hc1}.
\end{proof}

\begin{lem}   \label{kk7}
 For $p>2$ let $H=\GL_n(q)$, $H'=\SL_n(q)$ with $ep<n< ep^2$, or
 $H=\GU_n(q)$, $H'=\SU_n(q)$ with $dp<n<dp^2$. Let $G$ be a group such that
 $H'\le G\le H$. Then $G$ has no $\Syl_p$-regular character, unless $p=3$
 and $H=\GU_4(2)$.
\end{lem}

\begin{proof}
Suppose the contrary, and let $\chi$ be a $\Syl_p$-regular character of $G$.

Suppose first that $e>1,d>1$. Note that $G$ has a subgroup $X$, say,
isomorphic to $\SL_{ep}(q)\times\SL_{n-ep}(q)$, resp.,
$\SU_{dp}(q)\times\SU_{n-dp}(q)$, and $|G:X|_p=1$.
Let $S$ be a Sylow $p$-subgroup of $\SL_{n-ep}(q)$, resp., $\SU_{n-dp}(q)$.
Let $Y=S\times\SL_{ep}(q)$, resp., $S\times\SU_{dp}(q)$. 
As $|G:Y|_p=1$, by Lemma~\ref{hc1}, $r_{Y/S}^G(\chi)$ is a
$\Syl_p$-regular  character of $Y/S\cong \SL_{ep}(q)$, resp., $\SU_{dp}(q)$.
This contradicts Lemma~\ref{8ee}, unless, possibly, if $G=\SU_n(2)$ and $p=3$.
As $d>1$, this case does not occur.

Next, suppose that $e=d=1$, that is $p|(q-1)$ or $q+1$. Then we refine the
above argument. Set $D=\GL_p(q)$ or $\GU_p(q)$. Then $Y/S\cong D$. Set
$Y_1=G\cap Y$. Then $Y_1$ is normal in $Y$ and hence $O_p(Y_1)=Y_1\cap O_p(Y)$. 
As $Y/S\cong D$, it follows that $Y/O_p(Y)=D/O_p(D)=D/O_p(Z(D))$, and hence
$E:=Y_1/O_p(Y_1)$ is a non-central normal subgroup of $D/O_p(Z(D))$. By
Lemma~\ref{hc1}, $r_{Y_1/O_p(Y_1)}^G(\chi)$ is a $\Syl_p$-regular character
of $E=Y_1/O_p(Y_1)$. However, by Lemma~\ref{8ee}, $E$ has no $\Syl_p$-regular
character, unless $D=\GU_3(2)$ and $p=3$. So we are left with the case
$H=\GU_n(2)$, $p=3$ and $3<n<9$. 

The group $H=\GU_4(2)$ is excluded by assumption, so consider first
$H=\GU_5(2)$. As  $H=H'\times Z(H)$, it suffices to deal with $ G=\SU_5(2)$.
Suppose the contrary, and let $\chi$ be a $\Syl_3$-regular character of $G$.
Then $\chi(1)=|G|_p=243$. Let $\lambda$ be an \ir constituent of $\chi$,
so $\lambda(1)\leq 243$. If $g\in G$ is an element from class $3E$ then
$\lambda(g)>0$ unless $\lambda(1)=176$ or~220. As $\chi(g)=0$, there is a
constituent $\lambda_1$, say, of $\chi$ such that $\lam_1(1)\in\{176,220\}$.
Then $\lam(1)\leq 67=243-176$. Pick $h\in G$ from the class $3F$. Then
$\lam_1(h)>0$, and if $\lam(1)\leq 67$ then $\lam(h)>0$, unless $\lam(1)=10$.
So $\chi$ must have a constituent $\lambda_2$, say, of degree 10.
Then $\lambda_2(g)=4$. If $\lambda_1(1)=176$ then $\lambda_1(g)=-4$, and hence
$(\chi-\lambda_1-\lambda_2)(g)=0$. As $\lambda(g)>0$ if $\lambda(1)\leq 67$,
it follows that $\chi=\lambda_1+\lambda_2$, but then
$0=\chi(h)=\lambda_1(h)+\lambda_2(h)=3$, a contradiction. So
$\lambda_1(1)=220$, and the other constituents are of degree at most 23.
As $\lambda_1(g)=-5$, $\lambda_2(g)=4$, we have
$(\chi-\lambda_1-\lambda_2)(g)=-1$, in particular, for the other constituents
$\lambda$ of $\chi$ we have $\lambda(g)\leq 1$. By \cite{Atl}, this implies
$\lambda(1)=1$, and $\lambda$ must occur with multiplicity 1, whence
$\chi(1)=220+10+1=231$, a contradiction. 

Let $ H=\GU_6(2)$. By Lemma \ref{hc1}, it suffices to deal with $X:=\PGU_6(2)$.
Set $X'=\PSU_6(2)$. Then $|X'|_3=3^6=729$ and the \ir characters of $X'$ of
degree less than 616 are positive on class $3A$. In addition, $|X|_3=3^7=2187$,
and the \ir characters of $X$ of degree less than~2187 and not equal to 616 are
positive on class $3A$. Let $\chi$ be a $\Syl_3$-regular character of $X$ or
$X'$. Then the \ir character $\mu $ of degree 616 is a constituent of $\chi$.
Note that $\tau(3A)=-14$. If $\chi\in\Irr(X')$ then the sum of the other
constituents of $\chi$ is at most 113. By \cite{Atl}, they are of degree 1
or~22. The trivial character cannot occur with  multiplicity greater than~1,
so $113$ or 112 must be a multiple of 22, which is false. Let $\chi\in\Irr(X)$.
Note that the  multiplicity of  $\mu $ in $\chi$ is at most~3, and if $\mu$
occurs with multiplicity~3 then the  sum of the other constituents of $\chi$
is at most $2187-1848=339$. The \ir characters of degree at most 339 have
degrees 252,232,22,1, and all them as well as $\mu$ are positive at class $3C$.
This is a contradiction.
Suppose that $\mu$ occurs once. Then the sum of the other constituent values at
class $3A$ is $-14$. It follows that these constituents may only be of
degrees 770,252,232,22,1. Inspecting \cite{Atl}, one observes that all them
as well as $\mu$ are positive at class $3C$. This is a contradiction. So the
multiplicity of $\mu$ must be~2. Then the sum of the other
constituent values at class $3A$ is $-28$. Therefore, the degrees of the other
constituents may only be 770,560,385,252,232,22,1. Let $\nu$ be the character
of degree 385. Then $\nu(3A)=25$. If $(\chi,\nu)>0$ then the sum of the other
constituent values at class $3A$ is $-3$. The trivial character is the only one
whose value is at most  3.  As this cannot occur twice, we get a contradiction.
Therefore, $(\chi,\nu)=0$. As above, this contradicts 
$\chi(3C)=0$. This completes the analysis of the case with $n=6$.

Let $n=7$. Then $H'=\SU_7(2)$ contains a subgroup isomorphic to $\GU_6(2)$,
which contains a Sylow 3-subgroup of $H'$. So the result for this case follows
from $n=6$. In addition, $H=H'\cdot Z(H)$, so we are done by Lemma~\ref{hc1}.

Similarly, the result for $n=8$ follows from that with $n=7$.
\end{proof}
 
\begin{rem}
The group $\SU_4(2)$ has an \ir projective character of degree~81 (for $p=3$),
and hence $H=\GU_4(2)=\SU_4(2)\times Z(H)$ has a projective character of degree
$|H|_3=243$.
\end{rem}

\begin{thm}   \label{nn6}
 Let $p>2$ and $G$ be a group such that $\SL_n(q)\le G\le \GL_n(q)$, or
 $\SU_n(q)\le G\le\GU_n(q)$. Suppose that Sylow $p$-subgroups of $G/Z(G)$ are
 not cyclic. Then $G$ has no Steinberg-like character, unless $p=3$ and
 $G\in \{\SU_3(2),\GU_3(2),\SU_4(2),\GU_4(2)\}$.
\end{thm}

\begin{proof}
If $n\geq ep^2$ in the linear case and $n\geq dp^2$ in the unitary case then
the result follows from Proposition~\ref{t5r6} for $G/O_p(G)$ in place of $G$,
and then for $G$ in view of Lemma~\ref{hc1}.

If $ep<n< ep^2$ in the linear case and $dp<n< dp^2$ in the unitary case then
the result follows from Lemma~\ref{kk7}. If $2e\leq n\leq ep$ in the linear
case and $2d\leq  n \leq dp$ in the unitary case then the result follows from
Lemma~\ref{8ee}. If $n< 2e$ in the linear case and $n< 2d$ in the unitary
case then Sylow $p$-subgroups of $G/Z(G)$ are cyclic.  
\end{proof}

\begin{rem}
Proposition~\ref{t5r6} gives a better bound for $n$, but this does not yield
an essential advantage as the cases with $n=e(p+1)$ and $d(p+1)$ are not
covered by Proposition~\ref{t5r6}, and we have to use
Lemma~\ref{kk7} anyway.
\end{rem}

\subsection{The symplectic and orthogonal groups for $p>2$}

\begin{lem}   \label{lem:ort1}
 Let $G=\Sp_{2n}(q)$ ($q$ even, $n\ge2$, $(n,q)\neq (2,2)$),
 $G=\Spin_{2n+1}(q)$ ($q$ odd, $n\ge3$, $(n,q)\neq (3,3)$), or
 $G=\Spin_{2n}^\pm(q)$ ($n\ge4$). Suppose that Sylow $p$-subgroups of $G$ are
 abelian. Then $|G|_p<\mu_1(G)$.
\end{lem}

\begin{proof}
Let $S\in \Syl_p(G)$. As $S$ is abelian, we have $p>2$ and
$|S|\leq (q+1)^n$.
If $G=\Sp_{2n}(q)$, $q$ even, $n\ge2$, $(n,q)\neq (2,2)$, or $\Spin_{2n+1}(3)$
then $\mu_1(G)\geq (q^{n}-1)(q^n-q)/2(q+1)$ (see \cite[Table~II]{TZ96}). This
is greater than $(q+1)^n$.
If $G=\Spin_{2n+1}(q)$, $q>3$ odd, $n\ge3$, then $\mu_1(G)\geq
(q^{2n}-1)/(q^2-1)$. Again $\mu_1(G)>(q+1)^n$, whence the result.
The cases with $G=\Spin_{2n}^\pm(q)$, $n\ge4$, are similar, see
\cite[Thm.~7.6]{TZ96}.
\end{proof}

\begin{prop}   \label{a7a}
 Let $e$ be odd, $p>2$, and let $H=\Sp_{2n}(q)$ with $n>1$,
 $\GO_{2n+1}(q)$ with $n>2$, $\GO^+_{2n}(q)$ with $n>3$, or $\GO^-_{2n+2}(q)$
 with $n>2$. Suppose that $2e\leq n<ep^2$.
 Then $H$ has no $\Syl_p$-regular character.
\end{prop}

\begin{proof}
Let $S\in \Syl_p(H)$. By Lemma~\ref{nb8}(1), $S$ is conjugate to a Sylow
$p$-subgroup of a subgroup $H_1\cong\GL_n(q)$ of $H$.  By Lemmas~\ref{dd77},
\ref{u44}, \ref{8ee} and~\ref{kk7}, $\GL_n(q)$ for $2e\leq n<ep^2$ has no
$\Syl_p$-regular character, unless possibly when $n=2$.

Let $n=2$, so $H=\Sp_4(q)$, $e=1$ and $p|(q-1)$. Then $|H|_p=|q-1|_p^2$.
If $q$ is even then $\mu_1(G)=q(q-1)^2/2$ for $q>2$. This is greater than
$|H|_p$, whence the result. If $q$ is odd then $|H|_p\leq (q-1)^2/4$ and
$\mu_1(H)=(q^2-1)/2$. So again $|H|_p<\mu_1(H)$.
\end{proof}

\begin{prop}   \label{b7b}
 Let $e$ be even, $p>2$, and let $H=\Sp_{2n}(q)$ ($n>1$, $(n,q)\neq (2,2)$),
 $\GO_{2n+1}(q)$ ($q$ odd, $n>2$), or $\GO^\pm_{2n}(q)$ with $n>3$.
 Suppose that $2e\leq 2n<ep^2$. Then $H$ has no $\Syl_p$-regular character.
\end{prop}

\begin{proof}
Write $2n=ek+m$ with $m<e$, where $k>1$ is an integer. As $H$ contains a
subgroup $H_1$ with $(|H:H_1|,p)=1$, where $H_1\cong \Sp_{ke}(q)$ or
$\GO_{ke+1}(q)$, respectively, it suffices to prove the lemma for $2n=ke$.
Let $2n=ke$. By Lemma \ref{c5a}, a Sylow $p$-subgroup of $H$ is contained in
a subgroup isomorphic to $\GU_{k}(q^{e/2})$.
By Lemma~\ref{8ee} for $2<k\leq p$ and Lemma~\ref{kk7} for $p<k<p^2$
(with $d=1$ and $q^{e/2}$ in place of $q$), the group $\GU_{k}(q^{e/2})$ with
$(k,q^{e/2})\neq (3,2)$ has no $\Syl_p$-regular character, whence the claim.
(The exceptional case $H=\Sp_6(2)$, $p=3$ is considered below.)

Let $k=2$.  Then $H=\Sp_{2e}(q)$, and $p|(q^{e/2}+1)$. Then
$|H|_p=|q^{e/2}+1|_p^2$. If $q$ is even then $\mu_1(G)=(q^{e}-1)(q^e-q)/2(q+1)$
for $q>2$. This is greater than $|H|_p$, whence the result. If $q$ is odd
then $|H|_p\leq (q^{e/2}+1)^2/4$ and  $\mu_1(H)=(q^{e}-1)/2$. So
$|H|_p<\mu_1(H)$, whence the result.

A similar argument works if $H=\GO_{ke+1}(q)$ as well as for $H=\GO_{ke}^-(q)$
with $k$ odd, and for $H=\GO_{ke}^+(q)$ with $k>2$ even, except when
$H=\GO_8^+(2)$ and $e=2$.

Let $H=\GO_8^+(2)$ and $e=2$ so $p=3$. Then $|G|_3=243$ and the \ir characters
of degree less than~243 are of degrees $1,28,35,50,84,175,210$. Let $\chi$ be
a $\Syl_p$-vanishing character of degree 243 and $\lambda$ an \ir constituent
of $\chi$. By \cite{Atl}, $\lambda(3E)>0$ whenever $\lambda(1)\leq 243$.
This is a contradiction as $\chi(3E)=0$.

Let $k=2$ and $H=\GO^+_{2e}(q)$. Then $|H|_p= |q^{e/2}+1|_p^2<\mu_1(H)=(q^e-1)(q^{e-1}-1)/(q^2-1)$ for $q>2$ and $q=2,e>4$. (If $q=2,e=4$ then $p=5$, and $|H|_5=25<\mu_1(H)=28$.) So the result follows.
(The case $e=2$ has been examined above.)

Suppose that $H=\GO_{2n}^-(q)$ with $k$ even or $H=\GO_{2n}^+(q)$ with $k$ odd.
Then some Sylow $p$-subgroup of $H$ is contained in a subgroup $H_1$ isomorphic
to $\GO_{2n-e}^-(q)$ or $\GO_{2n-e}^+(q)$, respectively. Note that
$2n-e=(k-1)e$. For the groups $H_1$ the result has been proven above, except
for the cases where $k-1=1$ or $(k-1)e\leq 6$. However, if $k-1=1$ then Sylow
$p$-subgroups of $H_1$, and hence of $H$ are cyclic, and this case has been
examined in Propositions~\ref{prop:ir2} and~\ref{ms4}.
Let $(k-1)e\leq 6$. As $k-1>1$, we have $k=3,e=2$ as $e$ is even. Then
$H=\GO^-_{8}(q)$ and $\mu_1(H)=q(q^4+1)$ \cite{TZ96}. If $p>3$ then
$|H|_p=|q+1|_p^3$, otherwise $|H|_3=3|q+1|_3^3$. Then $|H|_p<\mu_1(H)$,
unless $q=2$.

Let $q=2$, $p=3$. Then $|G|_3=81$. By \cite{Atl} the irreducible characters of
degree less than $81$ are of degrees $1,26,52$. Therefore only these
characters can occur as irreducible constituents of a $\Syl_3$-regular
character $\chi$. However, the values of these characters at an element
$g\in G$ in class $9A$ are $1,2,1$, in particular, positive. As $\chi(g)=0$,
this is a contradiction.

Suppose that $G=\Sp_6(2)$ and $p=3$.
Then $|S|=81$. Let $\tau$ be an irreducible constituent
of $\chi$. Then $\tau(1)\leq 81$. Let $g\in G$ belong to the conjugacy
class $3C$ in notation of \cite{Atl}. By inspection of the character table of
$G$ one observes that $\tau(g)\geq 0$ whenever $\tau(1)\leq 81$. Therefore,
$\tau(g)=0$ for every irreducible constituent $\tau$ of $\chi$. This implies
$\tau(1)\in\{21,27\}$, see \cite{Atl}. However, such a character takes
positive values at the elements of class $3A$.
So this case is ruled out.\end{proof}

\begin{rem}
Let $G$ be the universal covering group of $\SO_8^+(2)$. One observes that $G$
has a $\Syl_5$-regular character, and has no  $\Syl_3$-regular characters. If
$H=\Sp_4(2)$ then $H$ has Steinberg-like characters for  $p=3$, both reducible
and irreducible.
\end{rem}

\section{Classical groups at $p=2$}   \label{sec:class p=2}
In this section we investigate $\Syl_p$-regular and Steinberg-like characters
of simple classical groups over fields of odd order $q>3$ at the prime $p=2$.

\subsection{Linear and unitary groups at $p=2$}
We first deal with the smallest case:

\begin{prop}   \label{prop:PSL2}
 Let  $q>3$ be odd.
 \begin{enumerate}[\rm(a)]
  \item Let $G=\PSL_2(q)$. Then $G$ has a reducible $\Syl_2$-regular character
   if and only if $q+1=2^k$ for some $k\ge3$ or if $q=5$.
  \item Let $G=\SL_2(q)$. Then $G$ has a reducible $\Syl_2$-regular character
   if and only if $q\pm 1$ is a $2$-power.
 \end{enumerate}
\end{prop}

\begin{proof}
(a) The 2-part of $|G|$ is $|q-1|_2$ if $q\equiv1\pmod4$ and $|q+1|_2$ else.
The smallest non-trivial character degree is $(q+1)/2$ in the first case,
$(q-1)/2$ in the second. It follows that there cannot be $\Syl_2$-regular
characters in the first case, unless $q=5$. In the second case, it follows
from the character table of $G$ that the sum of the trivial and the Steinberg
character is $\Syl_2$-regular when $q+1$ is a power of~2, and there are no
cases otherwise. If $q=5$ then there are two reducible $\Syl_2$-regular
characters of degree 4 by \cite{Atl}.

(b) Let $\chi$ be a reducible $\Syl_2$-regular character. Let
$1\neq z\in Z(G)$; then $\chi=\chi_1+\chi_2$, where $\chi_1(z)=\chi_1(1)$ and
$\chi_2(z)=-\chi_2(1)$. By Lemma~\ref{hc1} (with $P=G$ and $U=Z(G)$), $\chi_1$
is a $\Syl_2$-regular character for $G/Z(G)=\PSL_2(q)$. If $\chi_1$ is
irreducible then $q\pm 1$ is a 2-power by Proposition~\ref{prop:ir2}(2);
by (a), this is also true if $\chi_1$ is reducible. So $q\pm 1$ is a 2-power.

Then there are \ir characters $\chi_1,\chi_2$ such that $\chi_1+\chi_2$ is
2-vanishing of degree $2(q\pm 1)=|G|_2$. Indeed, using the character table
of $G$ one observes that there exist irreducible characters $\chi_1,\chi_2$
of $G$ that vanish at non-central 2-elements of $G$, and such that
$\chi_1(z)=\chi_1(1)$ and  $\chi_2(z)=-\chi_2(1)$. It follows that
$\chi_1+\chi_2$ is a  reducible $\Syl_2$-regular character.
\end{proof}

We recall that $\mu_3(G)$ denotes the third smallest degree of a non-trivial
irreducible representation of $G$.

\begin{lem}   \label{3rd}
 Let $G$ be a quasi-simple group such that $G/Z(G)\in\{\PSL_n(q), \PSU_n(q)\}$
 with $n\ge3$,  $q>3$ odd and $|Z(G)|_2=1$. Then $|G|_2<\mu_3(G)$.
\end{lem}

\begin{proof}
Let first $G/Z(G)=\PSL_3(q)$, $q>3$ odd. Then $\mu_1(G)=q(q+1)$ by
Table~\ref{tab:min SLn}, while $|G|_2=2|q-1|_2^2$ if $q\equiv1\pmod4$,
respectively $|G|_2=4|q+1|_2$ if $q\equiv3\pmod4$. Thus, $|G|_2<\mu_1(G)$
unless $q-1$ is a 2-power. In the latter case, $\mu_3(G)=(q^2-1)(q-1)$, and
our claim follows. Next, let $G/Z(G)=\PSL_4(q)$ with $q>3$ odd. Then
$\mu_3(G)=(q^3-1)(q-1)/2$, which is larger than $|G|_2\le 2(q-1)^3$ for
$q\geq 5$. Now assume that $G/Z(G)=\PSL_n(q)$, with $n\ge5$ and $q>3$ odd.
Then $\mu_3(G)= \frac{(q^n-1)(q^{n-1}-q^2)}{(q^2-1)(q-1)}$ (see
Table~\ref{tab:min SLn}) while $|G|_2\le (q-1)^{n-1}2^{n-1}$ when
$q\equiv1\pmod4$, and $|G|_2\le (q+1)^{\lfloor n/2\rfloor} 2^{n-1}$ when
$q\equiv3\pmod4$. Again, the claim follows.
\par
Let $G/Z(G)=\PSU_3(q)$, $q>3$ odd. Then $\mu_1(G)=q(q-1)$ by
\cite[Table~V]{TZ96}, while $|G|_2=2|q+1|_2^2$ if $q\equiv3\pmod4$,
respectively $|G|_2=4|q-1|_2$ if $q\equiv1\pmod4$. Thus, $|G|_2<\mu_1(G)$
unless $q+1$ is a 2-power. In the latter case, $\mu_3(G)=(q^2-q+1)(q-1)$, and
our claim follows. Now let $G/Z(G)=\PSU_4(q)$, $q>3$ odd. Then
$\mu_3(G)= \frac{(q^2-q+1)(q^2+1)}{2}$.
Suppose first that $4|(q+1)$.  We have $|G|_2=|\PSU_4(q)|_2\leq 2(q+1)^3$,
whereas $\mu_3(G)=(q^2+1)(q^2-q+1)/2$. Then $|G|_2<\mu_3(G)$. Suppose
now that $4|(q-1)$. Then $|G|_2\leq 2(q-1)^2$ which is less than
$\mu_1(G)=(q^4-1)/(q+1)$.
Now assume that $G/Z(G)=\PSU_n(q)$ with $n\ge5$ odd. Here
$\mu_3(G)= \frac{(q^n+1)(q^{n-1}-q^2)}{(q^2-1)(q+1)}$, while
$|G|_2\le (q-1)^{(n-1)/2}2^{n-1}$ if $q\equiv1\pmod4$, respectively
$|G|_2\le (q+1)^{n-1}2^{n-1}$ if $q\equiv3\pmod4$. The claim follows.
Finally, assume that $G/Z(G)=\PSU_n(q)$ with $n\ge6$ even. Here
$\mu_3(G)= \frac{(q^n-1)(q^{n-1}+1)}{(q^2-1)(q+1)}$, while the 2-part of
$|G|$ is bounded above as given before. Again we can conclude.
\end{proof}

\begin{prop}   \label{44d}
 Let $G$ be quasi-simple with $G/Z(G)\in\{\PSL_n(q), \PSU_n(q)\}$ with $n\ge3$
 and $q>3$ odd.
 \begin{enumerate}
  \item[\rm(a)] If $n=3,4$ then $G$ has no $\Syl_2$-regular character.
  \item[\rm(b)] If $n\ge5$ then $G$ has no Steinberg-like character for $p=2$.
 \end{enumerate}
\end{prop}

\begin{proof}
By Lemma~\ref{hc1} (with $P=G$), it suffices to prove the result in the case
where $|Z(G)|_2=1$. So we assume this, and then $|G|_2$ equals the order of
a Sylow $2$-subgroup of $G/Z(G)$.

Let first $G/Z(G)=\PSL_n(q)$, $q>3$ odd. Let $\chi$ be a $\Syl_2$-regular
character for $G$. By Lemma~\ref{3rd}, $\chi(1)<\mu_3(G)$, and hence the
non-trivial irreducible constituents of $\chi$ are of degree $(q^n-1)/(q-1)$ or
$(q^n-q)/(q-1)$, see Table~\ref{tab:min SLn}. It is known that the irreducible
characters of degree $(q^n-1)/(q-1)$ are induced characters $\lam^G$, where
$\lam\neq 1_P$ is a one-dimensional character of the stabiliser $P$ of a line
of the underlying space for $\GL_n(q)$, while the irreducible character of
degree $(q^n-q)/(q-1)$ is the unique non-trivial constituent $\tau$ of the
permutation character $1_P^G=\tau+1_G$ on $P$.
\par
Let $n\ge5$ and let $g\in\SL_n(q)$ be a block-diagonal matrix, with an
$(n-2)\times(n-2)$-block
corresponding to a primitive element of $\FF_{q^{n-2}}$ with determinant~1,
and a $2\times2$-block corresponding to an element of order~$q+1$. Since $g$
has no eigenvalue in $\FF_q$, no conjugate of $g$ is contained in $P$, so all
induced characters from $P$ to $G$ vanish on $g$. In particular $\lam^G(g)=0$
and $\tau(g)=-1$. Note that the image $\bar g\in G$ of $g$ has even order,
so $\chi(g)=0$ if $\chi$ is Steinberg-like. Write
$\chi=x_1 1_G+x_2\tau+\Lambda$, where $\Lambda$ is a sum of $x_3$ induced
characters of degree $(q^n-1)/(q-1)$, with suitable $x_i\ge0$. Evaluating on
$g$ we see that $x_1=x_2$, but then $\chi(1)=(x_1+x_3)(q^n-1)/(q-1)$ is
divisible by some odd prime, so cannot equal the 2-power $|G|_2$. This proves
part~(b) for $G/Z(G)=\PSL_n(q)$.
\par
Now assume that $n=4$. Then an easy estimate shows that when $q-1$ is not a
2-power, so $|q-1|_2\le(q-1)/3$, and $q\ne7$, then $|G|_2<\mu_1(G)$. So we may
assume that in addition either $q=7$ or $q-1$ is a power of~2. For $q\ne7$ let
$g$ be the 2-element
$$g=\begin{pmatrix} 0&1&0&0\\ a&0&0&0\\ 0&0&0&1\\ 0&0&a\up &0\end{pmatrix}
  \in\SL_4(q),$$
where $a\in\FF_q^\times$ is a 2-element of order $q-1$. Observe that
again $g$ is not conjugate to an element of $P$, thus $\lam^G(h)=0$ and
$\tau(h)=-1$. As $g$ is a 2-element and $\chi$ is $\Syl_2$-regular, we have
$\chi(g)=0$. We may now argue as above to conclude. When $q=7$ then the
candidate characters have degrees 1, 399 and~400, while $|G|_2=2^9=512$, so
clearly there can be no $\Syl_2$-regular character.
\par
Now consider the case when $G/Z(G)=\PSL_3(q)$. The proof of Lemma~\ref{3rd}
shows that $\mu_1(G)>|G|_2$ unless $q-1$ is a 2-power. In the latter case the
possible constituents of $\chi$ can have degrees $1,q^2+q,q^2+q+1$, while
$|G|_2=2(q-1)^2$. Clearly at most one of the degrees $q^2+q,q^2+q+1$ can
contribute to $\chi(1)$, but then necessarily $q=5$. But in that case the
character table shows that there's no $\Syl_2$-regular character. This
completes the proof of~(a) when $G/Z(G)=\PSL_n(q)$.
\par
Now let $G/Z(G)=\PSU_n(q)$ with $q>3$ odd, and let $\chi$ be a $\Syl_2$-regular
character of $G$. According to Lemma~\ref{3rd}, $\chi(1)<\mu_3(G)$, and hence
the non-trivial irreducible constituents of $\chi$ are of degree
$(q^n-(-1)^n)/(q+1)$ or $(q^n+(-1)^nq)/(q+1)$, see \cite[Table~V]{TZ96}. The
first of these are semisimple characters lying in the Lusztig series of an
element $s$ of order~$q+1$ in the dual group $G^*=\PGU_n(q)$ with centraliser
$C_{G^*}(s)\cong\GU_{n-1}(q)$, the second is a unipotent character, $\tau$ say,
corresponding to the character of the Weyl group $S_n$ parametrised by the
partition $(n-1,1)$. Let $g\in\SU_n(q)$ be a regular element of even order in
a maximal torus $T$ of order $(q^2-1)(q^{2n-2}-(-1)^{2n-2})/(q+1)$, see
\cite[Lemma~3.1(a)]{LM15}. Then no conjugate of the dual maximal torus $T^*$
contains $s$, so the characters in $\cE(G,s)$ vanish on $g$ (see
e.g.~\cite[Prop.~6.4]{LMS14}). If $\chi$ is Steinberg-like, then $\chi(g)=0$.
As $\tau$ is unipotent, its value on $g$ is (up to sign) the same as $\psi(h)$
where $\psi\in\Irr(S_n)$ is labelled by $(n-1,1)$ and $h$ is a permutation of
cycle shape $(n-2,2)$, see \cite[Prop.~3.3 and remark before Prop.~4.2]{LM15}.
The Murnaghan--Nakayama rule gives that $\psi(h)\in\{\pm1\}$, so
$\tau(g)\in\{\pm1\}$. We may now argue as in the first part to conclude that
$\chi$ cannot be Steinberg-like, thus completing the proof of~(b).
\par
Next, assume that $G/Z(G)=\PSU_4(q)$ with $q>3$ odd. If $q+1$ is not a
power of $2$ and $q\ne5,9$ then $|G|_2<\mu_1(G)$, as $|q+1|_2\le (q+1)/3$. So
now assume that $q+1$ is a power of $2$, and hence in particular
$q\equiv3\pmod4$. Then $|G|_2=2(q+1)^3$, while the three smallest character
degrees are $1,(q^4-1)/(q+1),(q^4+q)/(q+1)$, with the trivial character
occurring at most once. It is easily seen that there is no non-negative
integral solution for a possible decomposition of $\chi$. When $q=5$ then the
three smallest degrees are 1,\,104,\,105, while $|G|_2=128$; if $q=9$ then
the three smallest degrees are 1,\,656,\,657 while $|G|_2=512$; so in neither
case can there be $\Syl_2$-regular characters either.
\par
Finally, when $G/Z(G)=\PSU_3(q)$ then again the proof of Lemma~\ref{3rd} shows
that $q+1$ must be a 2-power. Here the possible constituents of $\chi$ have
degrees $1,q^2-q,q^2-q+1$, while $|G|_2=2(q+1)^2$. Again an easy consideration
shows that at most the case $q=7$ needs special attention. But there the
existence of $\Syl_2$-regular characters can be ruled out from the known
character table.
\end{proof}

We now treat the case $q=3$, which is considerably more delicate.

\begin{lem}   \label{sp43}
 Let $G=\PSL_3(3)$ or $\PSU_3(3)$. Then $G$ does not have reducible
 $\Syl_2$-regular characters.
\end{lem}

\begin{proof}
For $G=\PSL_3(3)$ we have $|G|_2=16$, and all irreducible characters of degree
less than~16 take non-negative values on class 4A, so there are no
$\Syl_2$-regular characters. For $G=\PSU_3(3)$ we have $|G|_2=32$, and all
irreducible characters have degree at most that large. Since the smallest
non-trivial character degree is~6, those of degrees 27 and~28 cannot be
constituents of a $\Syl_2$-regular character $\chi$. Thus, we need to
consider the characters of degrees 1,\,6,\,7,\,14,\,21. Clearly those of
degree~21 cannot occur either. As $32\equiv4\pmod7$ we see that the character
of degree~6 has to appear at least three times, but then the values on elements
of order~4 give a contradiction.
\end{proof}

\begin{rem}
$\PSL_3(3)$ and $\PSU_3(3)$ both have irreducible $\Syl_2$-regular
characters, see Proposition~\ref{prop:ir2}.
\end{rem}

\begin{lem}   \label{psl43}
 Let $G=\PSL_4(3)$ or $\PSU_4(3)$. Let $\chi$ be a $\Syl_2$-vanishing
 character of $G$. Then $l_2(\chi)\geq4$.
\end{lem}

\begin{proof}
Suppose the contrary. Note that $\chi$ is reducible. As $|G|_2=128$, we have
$\chi(1)\leq |G|_2\cdot 3= 384$. Let $\tau$ be an \ir constituent of $\chi$
of maximal degree. For all numerical data see \cite{Atl}.

Let $G=\PSL_4(3)$. Then $\tau(1)<351$ (otherwise $\tau(1)=351$, so
$(\chi-\tau)(1)=33$, and then $\chi(2B)>0$, which is false).

Let $\mu\in\Irr(G)$ and $\mu(1)<351$. Then $\mu(4B)\geq 0$ unless $\mu(1)=90$,
and $\mu(4B)\neq 0$ unless $\mu(1)=52$ or $260$. Let $\mu(1)=90$. Then
$(\chi,\mu)>0$. Indeed, otherwise the \ir constituents of $\chi$ are of degree
$52$ or $260$, which implies $\chi(2A)>0$, a contradiction.

It follows that $(\chi-\mu)(1)\leq 294$. Let $\sigma\in\Irr(G)$ with
$\sigma(1)=39$. The \ir characters of $G$ of degree at most 294 and distinct
from $\si$  are non-negative at $2A$.
In addition, $\si(2A)=-1$, and $\mu(2A)=10$. As $\chi(2A)=0$, it follows that
$(\chi,\si)\geq 10$. Then $\chi(1)\geq \mu(1)+10\si(1) >384$, a contradiction.

Let $G=\PSU_4(3)$ and $\mu\in\Irr(G)$ with $(\chi,\mu)>0$.  The \ir
characters of degree at most $384$ are of degree at most $315$.
If $\tau(1)=315$ then $\mu(1)\leq 69$.  Then $\tau(2A)>0$ and $\mu(2A)>0$,
a contradiction.

Suppose that $\tau(1)=280$. Then $\mu(1)\leq 104$, but then we obtain a
positive value on class $2A$. The same consideration rules out $\tau(1)=210$.

Suppose that $\tau(1)=189$. It occurs once as otherwise $1_G$ occurs 6 times,
which is false as $(\chi,1_G)\leq 3$. Then $\mu(1)\neq 140,90$, so $\mu(4A)>0$,
$\tau(4A)>0$, a contradiction. No more option exists, as the \ir characters
of degree less than 189 are positive on $2A$.
\end{proof}

\begin{lem}   \label{nr1}
 Let $\PSL_4(3)\le G\le\PGL_4(3)$ or $\PSU_4(3)\le G\le \PGU_4(3)$.
 Let $\chi$ be a $2$-vanishing character of $G$. Then $l_2(\chi)\geq 4$.
\end{lem}

\begin{proof}
By Lemma~\ref{ga1}(a), $\chi=\psi^G$, where $\psi$ is a proper character of
$G'$. By inspection of the character table of $G'$, see \cite{Atl}, it is
easily checked that the conjugacy class of any element $g\in G'$ of $2$-power
order is $G$-invariant. Then, by Lemma~\ref{ga1}(b), $\psi(g)=0$ for every
2-element $g\neq 1$ of $G'$. This means that $\psi$ is $\Syl_2$-vanishing.
By Lemma \ref{psl43}, $l_2(\psi)\geq 4$. Then $l_2(\psi^G)\geq 4$.
\end{proof}

\begin{lem}   \label{pi4}
 Let $G=\PGL_4(3)$, and let $\chi$ be a $2$-vanishing character of $G$. Let
 $\eta_1,\ldots,\eta_k$ be the \ir constituents of $\chi$ disregarding
 multiplicities, and $\eta=\eta_1+\cdots +\eta_k$. Then $\eta(1)\geq 2|G|_2$.
\end{lem}

\begin{proof}
Let $G'=\PSL_4(3)$. Then $|G'|_2=128$ and thus $2\cdot |G|_2=512$.
Suppose the contrary. Then $\eta(1)< 512$. We can assume that
$\eta_i(1)\geq \eta_j(1)$ for $1\leq i< j\leq k$. Note that $k>1$,
otherwise $\chi= a\eta_1$ for some $a$, and hence $\eta_1$ is a $2$-vanishing
character of $G$ and $\eta_1(1)$ is a multiple of $|G|_2$. By \cite{Atl},
$G$ has no character of degree at most 512 with this property.

By \cite{Atl}, we have $\eta_1(1)\leq 468$. Note that all \ir characters of
$G'$ of degree at most 468 extend to $G=G'\cdot 2$ except $\chi_{11},\chi_{12}$
of degree~260, $\chi_9,\chi_{10}$ of degree~234, $\chi_6,\chi_7$ of degree~65
and $\chi_2,\chi_3$ of degree 26. The corresponding characters of $G$ are of
degrees $520$, $468$, $130$ and $52$, respectively.
Let $\chi=\sum a_i\eta_i$, where $a_i>0$ are integers.

(i) Suppose that $\eta_1(1)=468$. Then $\sum_{i>1}\eta_i(1)\leq 44$. Computing
$\chi(4B)$ we get a contradiction (as $\eta_1(4B)=0$ and
$\sum_{i>1}a_i\eta_i(4B)>0$ and $k>1$).

(ii) Suppose $\eta_1(1)=416$. If $\eta_2(1)\leq 90$ then
$\sum_{i>2}\eta_i(1)\leq 6$, whence $k=3$ and $\eta_3=1_G$. Computing
$\chi(2B)$ we get a contradiction.

So $\eta_2(1)\leq 52$. If $\eta_2(1)= 52$ then $\sum_{i>2}\eta_i(1)\leq 44$.
Computing $\chi(4B)$ we get a contradiction, unless $k=2$ and $\eta_2(4B)=0$
(that is, $\eta_2=\chi_5$ in \cite{Atl}). In this case computing $\chi(4C)$
gives a contradiction. So $\eta_i(1)\leq 39$ for $i>1$. This violates
$\chi(2B)=0$.

(iii) Let $\eta_1(1)=390$. Then $\eta_2(1)\leq 90$. If $\eta_2(1)=90$ then
$\eta_3(1)\leq 32$, whence $k=1$ and $\eta_3(1)=1$. This conflicts with
$\chi(4A)=0$. If $\eta_2(1)\leq 52$ then computing $\chi(4B)$ yields a
contradiction.

(iv) Let  $\eta_1(1)=351$ so $\sum_{i>1}\eta_i(1)\leq 161$. If $\eta_2(1)=130$
then $\sum_{i>1}\eta_i(1)\leq31$, and hence $\eta_3(1)=1$, which yields a
contradiction with  $\chi(20A)=0$.
Let $\eta_2(1)\leq 90$. Then $\eta_i(1)\leq 71$ for $i>2$. Then
$\eta_i(2A)+\eta_i(2B)>0$ for $i=1,\ldots,k$, which violates
$\chi(2A)+\chi(2B)=0$.

(v) Let  $\eta_1(1)=260$. Then $\eta_1(2A)\geq 0$ and the only \ir character
$\lam$ of degree less than~260 with negative value at $2A$ has $\lam(1)=39$.
It follows that $(\chi,\lam)\geq 0$, and then $\eta_i(1)\leq 213$ if $i>1$ and
$\eta_i\neq \lam$. Note that $\lam(8A)=1$ and $\eta_1(8A)=0$. As $\chi(8A)=0$,
it follows that a character of degree 130 occurs in $\chi$, which implies
$\eta_2(1)=130$. Then we get contradiction to $\chi(2A)+\chi(4B)=0$.

(vi)  Let  $\eta_1(1)\leq 234$. Computing $\chi(2A)+\chi(4B)$ leads to a
contradiction.
\end{proof}

\begin{lem}   \label{pu4}
 Let $G=\PGU_4(3) $, and $\chi$ a $2$-vanishing character of $G$.
 Let $\eta_1,\ldots,\eta_k$ be the \ir constituents of $\chi$ disregarding
 multiplicities, and $\eta=\eta_1+\cdots +\eta_k$. Then $\eta(1)\geq 2|G|_2$.
\end{lem}

\begin{proof}
Let $G'=\PSU_4(3)$. Then $|G'|_2=128$, $|G|_2=512$ and $2\cdot |G|_2=1024$.
Suppose the contrary. Then $\eta(1)< 1024$. We can assume that
$\eta_i(1)\geq \eta_j(1)$ for $1\leq i< j\leq k$.

Note that $\chi\cdot \tau=\chi$ for every linear character $\tau$ of $G$.
Therefore, $\eta_i \tau$ is a constituent of $\eta$. Let $g\in G\setminus G'$.
Then $\eta_i\tau=\eta$ implies $\eta_i(g)=0$.

By \cite{Atl}, if $630\neq \eta_i(1)>420$ then $\eta_i(4E)\neq 0$ or
$\eta_i(4G)\neq 0$; it follows that $\eta$  must contain at least 2
representations of the same degree, which contradicts $\eta(1)\leq 1024$.

So $\eta_i(1)$ either equals 630 or $\eta_i(1)\leq 420$.  By \cite{Atl},
$\eta_i(2A)+\eta_i(4B)>0$ for these $\eta_i$, unless $\eta_i(1)=210$. This
violates $\chi(2A)+\chi(4B)=0$ unless $k=1$ and $\eta_1(1)=210$.
Then $\chi(2A)>0$, a contradiction.
\end{proof}

\begin{lem}   \label{gn5}
 Let $H=H_1\times \cdots \times H_n$, where
 $H_1\cong \cdots \cong H_n\cong \PGL_4(3)$ or $\PGU_4(3)$. Let $\chi$ be
 a $2$-vanishing character of H. Then $\chi(1)\geq 2^{n+1}|H|_2$.
\end{lem}

\begin{proof}
By Lemma \ref{nr1}, the claim holds for $n=1$, so by induction we can assume
that it is true for $X:=H_2\times \cdots \times H_n$. By Lemma \ref{nn2},
$\chi=\sum_i \eta_i\sigma_i$, where $\eta_i\in \Irr(H_1)$, $\sigma_i$ are
$2$-vanishing characters of $X$ and $\chi'=\sum_i l_2(\sigma_i)\eta_i$ is a
2-vanishing character of $H_1$. By induction, $\sigma_i(1)\geq 2^{n}|X|_2$.
By Lemmas~\ref{pi4} and~\ref{pu4} applied to $\chi'$, we have
$\sum _i\eta_i(1)\geq 2|H_1|_2$, so  $\chi(1)\geq 2^{n+1}|H|_2 $ by Lemma~\ref{ga1}.
\end{proof}

\begin{prop}   \label{gl4n}
 Let $n>1$ and $G =\GL_{4n}(3)$ or $\GU_{4n}(3)$. Let $\chi$ be a
 $2$-vanishing character of $G$. Then $l_2(\chi)\geq 4$.
\end{prop}

\begin{proof}
Let $X$ be the direct product of $n$ copies of $\GL_4(3)$ or $\GU_4(3)$. Let
$\chi$ be a 2-vanishing character of $X$. By Lemmas~\ref{hc1} and~\ref{gn5},
$\chi(1)\geq 2^{n+1}|X|_2$.

Let  $Y=X\cdot S_n$, the semidirect product, where $S_n$ acts on $X$ by
permuting the factors. Then $Y$ contains a Sylow 2-subgroup of $G$. Let
$M=X\cdot S$, where $S\in \Syl_2(S_n)$, so the index $|G:M|$ is odd.
Note that $|G|_2=|X|_2\cdot |S_n|_2$. As $|S_n|_2\leq 2^{n-1}$ (see the proof
of Proposition~\ref{so1}), the result follows for these groups.
\end{proof}

\begin{thm}   \label{thm:GL p=2}
 Let $p=2$, $m>3$ and $G$ one of $\GL_m(3)$, $\SL_m(3)$, $\PSL_m(3)$,
 $\GU_m(3)$, $\SU_m(3)$, $\PSU_m(3)$. Then $G$ has no Steinberg-like character.
 Moreover, if $\chi$ is a $2$-vanishing character of $G$ then $l_2(\chi)\geq 4$.
\end{thm}

\begin{proof}
Let first $G=\GL_m(3)$ or $\GU_m(3)$. For $m\equiv 0\pmod 4$ the result is
stated
in Proposition~\ref{gl4n}. Let $m=4n+l$, where $1\leq l<4$, and $H=\GL_{4n}(3)$
or $\GU_{4n}(3)$. Let $S_0$ be a Sylow $2$-subgroup of $\GL_l(3)$ or $\GU_l(3)$
and set $U=H\times S_0$. Then $U$ contains a Sylow 2-subgroup of $G$.
Therefore, $l_2(\chi)= l_2(\chi|_U)$. By Lemma~\ref{hc1}, if $\nu$ is a
2-vanishing character of $U$ then $l_2(\nu)= l_2(\mu)$ for some
2-vanishing character $\mu$ of $H$. So $l_2(\nu)\geq 4$ by
Proposition~\ref{gl4n}. So the result follows for these groups.

For $G=\SL_m(3)$ or $\SU_m(3)$ the result follows from the above and
Lemma \ref{ob4}. For $G=\PSL_m(3)$ or $\PSU_m(3)$ the statement follows
from the above and Lemma \ref{hc1}.
\end{proof}

\subsection{Orthogonal and symplectic groups at $p=2$}
Let $V$ be the natural module for $H=\Sp_{2n}(q)$, $q$ odd, and for $g\in G$
let $d(g)$ be the dimension of the fixed point subspace of $g$ on $V$. Let
$\om_n$ denote the Weil character of $H$. By Howe \cite[Prop.~2]{Ho},
$|\om_n(g)|=q^{d(g)/2}$. Let $\om_n=\om_n'+\om_n''$, where
$\om_n',\om_n''\in\Irr(H)$ and $\om_n'(z)=-\om_n'(1)$ for $1\neq z\in Z(H)$.

\begin{lem}   \label{lem:3wr}
 \begin{enumerate}[\rm(a)]
  \item Let $h\in H$ be semisimple such that $h$ and $zh$ fix
   no non-zero vector on V. Then $|\om_n''(h)|\leq 1$.
  \item Let $V=V_1\oplus V_2$, where $V_1$ is a non-degenerate subspace
   of dimension $2$, let $g\in H$ be an element such that $gV_i=V_i$, $i=1,2$,
   $g|_{V_1}=-\Id$, and $g$ and $zg$ fix  no non-zero vector on $V_2$. Then
   $|\om_n''(g)|\geq (q-1)/2$.
  \item Let $q>3$. Then $\om_n''$ is not constant on the $2$-singular
   elements of $\PSp_{2n}(q)$.
 \end{enumerate}
\end{lem}

\begin{proof}
(a) We have $\om_n(h)=\om_n'(h)+\om_n''(h)$ and
$\om_n(zh)=-\om_n'(h)+\om_n''(h)$. Therefore, by \cite[Prop.~2]{Ho},
$2\geq |\om_n(h)+\om_n(zh)|=|2\om_n''(h)|$, whence the claim.

(b) We have $|\om_n(g)|=1$ and $|\om_n(zg)|=q$ by \cite[Prop.~2]{Ho}. Then
$q-1\leq |\om_n(g)+\om_n(zg)|=|2\om_n''(g)|$, whence the claim.

(c) Choose $g$ as in~(b) and $h$ to be an element stabilising $V_1,V_2$ such
that $h$ coincides with $g$ on $V_2$ and the matrix of $h$ on $V_1$ is similar
to $\begin{pmatrix}0&1\\ -1&0\end{pmatrix}$. Then $h$ is a 2-singular element
satisfying~(a), and hence $|\om_n''(h)|\leq 1$.
Let $\overline{h}$ and $\overline{g}$ be the images of $h,g$ in $H/Z(H)$.
Then $\overline{h}$ and $\overline{g}$ are 2-singular elements of
$\PSp_{2n}(q)$. As $Z(H)$ is in the kernel of $\om_n''$, this can be viewed
as a character of $H/Z(H)$. As $(q-1)/2$ is greater than 1 for $q>3$, (c)
follows.
\end{proof}

\begin{prop}   \label{prop:symplectic p=2}
 Let $G=\PSp_{2n}(q)$, with $q>3$ odd and $n\ge2$. Then $G$ has no
 $\Syl_2$-regular characters.
\end{prop}

\begin{proof}
We have
$$|\PSp_{2n}(q)|_2=\begin{cases}
  |q-1|_2^n \cdot 2^{n-1}\cdot  |n!|_2& \text{if 4 divides $q-1$},\\
  |q+1|_2^n \cdot  2^{n-1}\cdot  |n!|_2& \text{if 4 divides $q+1$}.
  \end{cases}$$
Let $k$ be minimal with $n\le2^k$, then as $|n!|_2\le|2^k!|_2=2^{2^k-1}$ we
have
$$|\PSp_{2n}(q)|_2\le\begin{cases}
  |q-1|_2^n\cdot 4^{n-1}& \text{if $4|(q-1)$},\\
  |q+1|_2^n\cdot 4^{n-1}& \text{if $4|(q+1)$}.
\end{cases}$$
On the other hand $\mu_3(G)=(q^n-1)(q^n-q)/(2(q+1))$ by \cite[Thm.~5.2]{TZ96},
and this is larger than $|G|_2$, unless $n=2$ and $q=5,7$. Let's set aside
these cases for a moment. Then otherwise if $\chi$ is $\Syl_2$-regular, the
constituents of $\chi$ are either Weil characters or the trivial character.
Now note that a Weil character of $\Sp_{2n}(q)$ of degree $(q^n\pm1)/2$
has the centre in its kernel if and only if its degree is odd. So, the
non-trivial constituents of $\chi$ have degree~$(q^n-1)/2$ if $q\equiv3\mod4$
and $n$ is odd, and $(q^n+1)/2$ else. According to Lemma~\ref{lem:a51} the
trivial character occurs at most once in $\chi$. As $(q^n\pm1)/2$ is never a
power of~2 for $n\ge2$ and odd $q$ (consider a Zsigmondy prime divisor), the
trivial character must occur exactly once. Let $\psi_1,\psi_2$ denote the two
Weil characters of $G$, interchanged by the outer diagonal automorphism
$\gamma$ of $G$. Observe that $\gamma$ is induced by an element of
$\GL_{2n}(q)$ and thus fixes all involution classes of $G$.
Let $g\in G$ be an involution and write $a:=\psi_1(g)=\psi_2(g)$. Then
$\chi(g)= ma+1$, where $m$ is the number of non-trivial constituents of $\chi$.
As necessarily $m>1$ (compare the degrees) we see that $\chi(g)\ne0$, so
$\chi$ is not $\Syl_2$-regular.
\par
We now discuss the two exceptions.
For $G=\PSp_4(5)$, $|G|_2=2^6=64$ and all irreducible characters of degree
at most~64 take non-negative values on class 2B, so there is no $\Syl_2$-regular
character.
For $G=\PSp_4(7)$, $|G|_2=2^8=256$ and all irreducible characters of degree at
most~256 take positive values on class 8A, except for one of degree~175 which
takes value~$-1$, and one of degree~224 which takes value~0. As at most one of
those latter two characters could occur, and at most once, there can
be no $\Syl_2$-regular character for $p=2$.
\end{proof}

\begin{prop}   \label{prop:odd-orth p=2}
 Let $G=\Omega_{2n+1}(q)$ with $q>3$ odd and $n\ge3$. Then $G$ has no
 $\Syl_2$-regular characters.
\end{prop}

\begin{proof}
According to \cite[Thm.~6.1]{TZ96} we have $\mu_1(G)=(q^{2n}-1)/(q^2-1)$, which
is larger than $|G|_2$ unless either $n=3$ and $q=7$, or $n=4$ and $q=5,7$.
\par
For $G=\Omega_7(7)$ the only non-trivial character of degree less than
$|G|_2=2^{12}$ is the semisimple character of degree~2451 (see \cite{TZ96}).
Since the trivial character can occur at most once in a $\Syl_2$-regular
character, we see that no example can arise here.
For $G=\Omega_9(5)$ the only non-trivial character of degree less than
$|G|_2=2^{14}$ is the semisimple character of degree~16276 (see \cite{TZ96}).
Again, this does not lead to an example. For $G=\Omega_9(7)$ the only
non-trivial character of degree less than $|G|_2=2^{18}$ is the character of
degree~120100, and we conclude as before.
\end{proof}

\begin{prop}   \label{prop:even-orth p=2}
 Let $G=\POm_{2n}^\pm(q)$ with $q>3$ odd and $n\ge4$. Then $G$ has no
 $\Syl_2$-regular characters.
\end{prop}

\begin{proof}
The second smallest non-trivial character degree of $G=\POm_{2n}^+(q)$ is
given by $\mu_2(G)=(q^n-1)(q^{n-1}-1)/(q+1)/2$ (see \cite[Thm.~1.4]{NG}),
which is larger
than $|G|_2$ unless $(n,q)=(4,7)$. Leaving that cases aside for a moment,
we see that any $\Syl_2$-regular character of $G$ is a multiple of the
smallest non-trivial character, of degree $(q^n-1)(q^{n-1}+q)/(q^2-1)$, plus
possibly the trivial character. Arguing as in the case of symplectic groups
we see that such characters take non-zero value on involutions. For
$G=\POm_8^+(7)$ the constituents of a $\Syl_2$-regular character could have
degree 1, 17500, or 51300. No non-negative integral linear combination of
these three degrees, with $1_G$ appearing at most once, adds up to
$|G|_2=2^{16}=65536$.
\par
The second smallest non-trivial character degree of $G=\POm_{2n}^-(q)$ is
$\mu_2(G)=(q^n+1)(q^{n-1}+1)/(q+1)/2$ (see again \cite[Thm.~1.4]{NG}), which
is larger than $|G|_2$. We conclude as before.
\end{proof}

Again, we are left with the case that $p=2$, $q=3$.

\begin{lem}   \label{mi8}
 Let $G=\PSp_6(3)$, $\Omega_7(3)$ or $\POm_8^-(3)$. Then $G$ has no
 Steinberg-like character.
\end{lem}

\begin{proof}
For $G=\PSp_6(3)$ we have $|G|_2=2^9$, and all irreducible characters of $G$
of degree at most~512 take positive value on the class $4A$, see \cite{Atl}.
\par
Let $G=\Omega_7(3)$. Then $|G|_2=2^9$. All \ir characters of $G$ of at most
that degree are positive at the elements of conjugacy class $2B$ \cite{Atl}.
\par
Let $G=\POm_8^-(3)$.  By \cite{Atl}, $G$ has 8 irreducible characters
of degree at most $|G|_2=2^{10}$. All of them take positive values on
class 2A. So the result follows.

This implies the result.
\end{proof}

\begin{lem}   \label{o81}
 Let  $G = \PGO_8^+(3)$ or $\PSp_8(3)$. Let $\chi$ be a $2$-vanishing
 character of $G$, and $\eta_1,\ldots,\eta_k$ the \ir constituents of $\chi$
 disregarding their multiplicities.  Set $\eta=\eta_1+\cdots +\eta_k$.
 Then $\eta(1)\geq 2|G|_2$.
\end{lem}

\begin{proof}
Suppose first that $G = \PGO_8^+(3)$. Note that $|G'|_2=2^{12}$, so
$2\cdot |G|_2=2^{15}=32768$. Suppose the contrary. Then $\eta(1)<32768$. We
use notation from \cite{Atl}. There are 43 characters of $G'$ of degree less
than 32768, the maximal degree among them is 29120.

There is only one \ir character of $G'$ of degree less that 32768 that is
negative at $4A$ (this is of degree 9450), while all other are positive.
So it must be a constituent of $\eta|_{G'}$. This character extends to $G$,
so the other
constituents of $\eta$ are of degrees at most $32768-9450=23318$.

It follows that $\eta_i(1)\leq 18200$. In fact, $\eta_i(1)<18200$. Indeed,
if $\eta_i(1)= 18200$ and $\eta_j(1)=9450$ for some $i\neq j$ then $\eta_l$
for $l\neq i,j$ are of degree at most $23318-18200=5118$. These characters are
positive at  $2A$ (as well as those of degree $18200$ and 9450).
This violates $\chi(2A)=0$.

Thus, $\eta_i(1)<18200$, and hence $\eta_i(1)\leq 17550$. Furthermore,
computing the character table of $G$ by a program in the computer package
GAP, one observes that there are 4 distinct \ir characters of degree~17550,
and only one \ir character of this degree for $G'$. It follows that these 4
characters differ from each other by multiplication by a linear character.
As $|G/G'|=4$, one observes that $\chi\cdot \lam=\chi$ for every linear
character $\lam$ of $G$. Therefore, $\eta_i\cdot \lam$ must be a constituent
of $\chi$. So, if $\eta_i(1)= 17550$ then there are 3 more constituents of
$\eta$ of this degree, which contradicts the inequality $\eta(1)<32768$.

Thus, $\eta_i(1)<17550$ for $i=1,\ldots,k$. By \cite{Atl}, all such \ir
characters of $G'$, and hence of $G$, are positive at $4A$, which contradicts
$\chi(4A)=0$.

Let $G =\PSp_8(3)$. Note that $|G|_2=2^{14}$, so $2\cdot |G|_2=2^{15}=32768$.
Suppose the contrary. Then $\eta(1)<32768$. There are 19 \ir characters of
degree less than~32768. All such characters are positive at $4A$,
which violates $\chi(4A)=0$.
\end{proof}

\begin{lem}   \label{go4}
 \begin{enumerate}[\rm(a)]
  \item Let $H=H_1\times\cdots\times H_n$, where
   $H_1\cong\cdots\cong H_n\cong\PGO_8^+(3)$ or $\PSp_8(3)$. Let $\chi$ be a
 $2$-vanishing character of $H$. Then $\chi(1)\geq 2^{n}|H|_2$.
  \item Let $G=G_1\times\cdots\times G_n$, where
   $G_1\cong\cdots\cong G_n\cong\GO_8^+(3)$ or $\Sp_8(3)$. Let $\chi$ be a
   $2$-vanishing character of G. Then $\chi(1)\geq 2^{n}|G|_2$.
 \end{enumerate}
\end{lem}

\begin{proof} (a) If  $n=1$ then the result is contained in Lemma~\ref{o81}.  
So by induction we can
assume that it is true for $X:=H_2\times \cdots \times H_n$. By
Lemma~\ref{nn2}, $\chi=\sum_i\eta_i\sigma_i$, where $\eta_i\in \Irr(H_1)$,
$\sigma_i$ are $2$-vanishing characters of $X$ and
$\chi'=\sum l_2(\sigma_i)\eta_i$ is a 2-vanishing character of $H_1$.
By induction, $\sigma_i(1)\geq 2^{n-1}|X|_2$. By Lemma~\ref{o81} applied to
$\chi'$, we have $\sum _i\eta_i(1)\geq 2|H_1|_2$, so $\chi(1)\ge 2^{n}|H|_2$ by
Lemma~\ref{ga1}.

(b) This follows from (a) and Lemma~\ref{hc1}, as $Z(G)$ is a 2-group.
\end{proof}

\begin{lem}   \label{o82}
 Let $G =\GO_{8n}^+(3)$, $\Omega_{8n}^+(3)$, $\POm_{8n}^+(3)$, $\Sp_{8n}(3)$
 or $\PSp_{8n}(3)$. Then $G$ has no Steinberg-like character for $p=2$.
\end{lem}

\begin{proof}
Let $X$ be the direct product of $n$ copies of $\GO_8^+(3)$ or $\Sp_8(3)$.
Let $\nu$ be a 2-vanishing character of $X$. By Lemma~\ref{gn5},
$\nu(1)\geq 2^{n}|X|_2$.

Let $Y=X\cdot S_n$, the semidirect product, where $S_n$ acts on $X$ by
permuting the factors. Then $Y$ contains a Sylow 2-subgroup of $G$. Let
$M=X\cdot S$, where $S\in \Syl_2(S_n)$, so the index $|G:M|$ is odd.
Note that $|G|_2=|X|_2\cdot |S_n|_2$. As $|S_n|_2\leq 2^{n-1}$ (see the proof
of Proposition~\ref{so1}), the result follows for the groups
$\GO^+_{8n}(3)$ and $\Sp_{8n}(3)$.

For $G=\Omega_{8n}^+(3)$ the result follows from the above and Lemma~\ref{ob4}.
For $G=\POm_{8n}^+(3)$ or $\PSp_{8n}(3)$ the statement follows from the
above and Lemma \ref{hc1}.
\end{proof}

\begin{prop}   \label{or2}
 Let $m\geq 4$ and $G=\GO_{2m}^+(3)$ or $\Sp_{2m}(3)$. Then $G$ has no
 Steinberg-like character for $p=2$.
\end{prop}

\begin{proof}
For $m\equiv 0\pmod 4$ the result is stated in Lemma~\ref{o82}.
Let $m=4n+l$, where $1\leq l<4$, and $H=\GO^+_{8n}(3)$ or $\Sp_{8n}(3)$.
Let $S_0$ be a Sylow $2$-subgroup of $\GO^+_{2l}(3)$ or $\Sp_{2l}(3)$.
Set $U=H\times S_0$. Then $U$ contains a Sylow 2-subgroup of $G$. Let $\chi$
be a $2$-vanishing character of $G$. Therefore, $l_2(\chi)= l_2(\chi|_U)$.
By Lemma~\ref{hc1}, if $\nu$ is a 2-vanishing character of $U$ then
$l_2(\nu)= l_2(\mu)$ for some 2-vanishing character $\mu$ of $H$.
By Lemma~\ref{o82}, $l_2(\mu)\geq 2$. So $l_2(\nu)\geq 2$, and the
result follows.
\end{proof}

\begin{prop}   \label{ot6}
 Let $G=\GO_{2m}^-(3)$,  $m\geq 5$. Then $G$ has no Steinberg-like character
 for $p=2$.
\end{prop}

\begin{proof}
Let $m=4n+l$, where $1\leq l\leq 4$, and let $H=\GO^+_{8n}(3)$. Then $G$
contains a subgroup $D$ isomorphic to $H\times \GO_{2l}^-(3)$. Then one
concludes that $D$ contains
a Sylow 2-subgroup of $G$. Let $S_0$ be a Sylow $2$-subgroup of $\GO_{2l}^-(3)$.
Set $U=H\times S_0$. Then $U$ contains a Sylow 2-subgroup of $G$.
By Lemma~\ref{hc1}, if $\nu$ is a 2-vanishing character of $U$ then
$l_2(\nu)= l_2(\mu)$ for some 2-vanishing character $\mu$ of $H$.
So $l_2(\nu)\geq 2$ by Lemma~\ref{o82} and the result follows.
\end{proof}

\begin{prop}   \label{ot7}
 Let $G=\GO_{2m+1}(3)$, $m\geq 3$. Then $G$ has no Steinberg-like character
 for $p=2$.
\end{prop}

\begin{proof}
The case $m=3$ is dealt with in Lemma~\ref{mi8}. So we assume that $m>3$,
that is, $2m+1\geq 9$. Let $m=4n+l$, where $0\leq l\leq 3$, and let
$H=\GO^+_{8n}(3)$. Then  $G$ contains a subgroup $D$ isomorphic to
$H\times \GO_{2l+1}(3)$. Then $D$ contains a Sylow 2-subgroup of $G$.
Set $U=H\times S_0$, where $S_0$ is a Sylow 2-subgroup of $\GO_{2l+1}(3)$,
so $U$ contains a Sylow 2-subgroups of $G$.
By Lemma~\ref{hc1}, if $\nu$ is a 2-vanishing character of $U$ then
$l_2(\nu)= l_2(\mu)$ for some 2-vanishing character $\mu$ of $H$.
So $l_2(\nu)\geq 2$ by Lemma~\ref{o82}, and the result follows.
\end{proof}

\begin{thm}   \label{ot8}
 Let $G=\GO_{2m+1}(3)$ with $m\geq 3$, $\GO_{2m}^\pm(3)$ with $m\geq 4$, or
 $\Sp_{2m}(3)$ with $m\ge3$, and let $G'$ be the derived group of G. Let $H$
 be a group such that $G'\le H\le G$. Then $H$ and $H/Z(H)$ have no
 Steinberg-like character for $p=2$.
\end{thm}

\begin{proof}
For $H$ the statement follows from Lemma~\ref{mi8}, Propositions~\ref{or2},
\ref{ot6} and \ref{ot7} using Lemma~\ref{ob4}, and for $H/Z(H)$ from
Lemma~\ref{hc1}.
\end{proof}

We now collect our results to prove our main theorems from the introduction.

\begin{proof}[Proof of Theorem~$\ref{th1}$]
Assume that $G$ is a finite non-abelian simple group possessing a
Steinberg-like character $\chi$ with respect to a prime $p$. The cases when
$\chi$ is
irreducible have been recalled in Proposition~\ref{prop:ir2}. If Sylow
$p$-subgroups of $G$ are cyclic, then $(G,p,\chi)$ is as in
Proposition~$\ref{ms4}$. So we may now assume that Sylow $p$-subgroups of $G$
are non-cyclic. For $G$ alternating and $p$ odd there are no cases by
Theorem~\ref{thm:An} except for $A_6\cong\PSL_2(9)$ with $p=3$. The Steinberg-like
characters of sporadic groups are listed in Theorem~\ref{thm:spor}.
\par
Thus $G$ is of Lie type. The case when $p$ is the defining prime was handled
in \cite{PZ} and Propositions~\ref{prop:B3 defchar} and~\ref{prop:D4 defchar},
respectively. So now assume $p$ is not the defining prime for $G$. Groups of
exceptional Lie type were handled in Theorem~\ref{thm:exc}. For
classical groups of large rank with $p$ odd, our result is contained in
Proposition~\ref{t5r6}, the cases for $\PSL_n(q)$ and $\PSU_n(q)$ with $p>2$
are completed in Theorem~\ref{nn6}, those for the other classical groups in
Propositions~\ref{a7a} and~\ref{b7b}. Finally, the cases with $p=2$ are covered
by Proposition~\ref{prop:PSL2} for $G=\PSL_2(q)$, Proposition~\ref{44d} for
$\PSL_n(q)$ and $\PSU_n(q)$ with $q\ne3$, Theorem~\ref{thm:GL p=2} for
$\PSL_n(3)$ and $\PSU_n(3)$, Propositions~\ref{prop:symplectic p=2},
\ref{prop:odd-orth p=2} and~\ref{prop:even-orth p=2} for classical groups with
$q\ne3$, and Theorem~\ref{ot8} for the case that $q=3$.
\end{proof}

\begin{proof}[Proof of Theorem~$ \ref{th2}$]
The characters of projective $\overline{\FF}_pG$-modules of dimension $|G|_p$
are in particular Steinberg-like, so in order to prove this result we need to
go through the list given in Theorem~\ref{th1}(2)--(5). When Sylow $p$-subgroups
of $G$ are cyclic, the possibilities are given in Lemma~\ref{re56}(b). For
$G$ of Lie type in characteristic~$p$, see \cite[Thm.~1.1]{Z11}. The case~(4)
of Theorem~\ref{th1} is subsumed in statement~(1), and finally the alternating
groups for $p=2$ are discussed in Theorem~\ref{22a}.
\end{proof}


\end{document}